\documentclass[reqno, 12pt,a4letter]{amsart}
\usepackage{amsmath,amsxtra,amssymb,latexsym, amscd,amsthm}
\usepackage[mathscr]{euscript}
\usepackage{mathrsfs}
\usepackage[english]{babel}
\usepackage[active]{srcltx}
\usepackage{enumerate}
\usepackage{color}
\usepackage{mathtools,stmaryrd}
\usepackage[T1]{fontenc}
\usepackage{tikz}
\usepackage[utf8]{inputenc}
\usetikzlibrary{calc}
\usetikzlibrary{math}
\usetikzlibrary{shapes.geometric,positioning}
\usepackage{textcomp}

\newtheorem {theorem}{Theorem}[section]
\newtheorem {proposition}{Proposition}[section]
\newtheorem {lemma}{Lemma}[section]

\theoremstyle{definition}
\newtheorem{definition}{Definition}[section]

\newcommand{\rank}{{\rm rank}}

\newcommand{\dist}{{\rm dist}}

\newcommand{\supp}{{\rm supp }}

\newcommand{\vv}{{\rm v}}
\newcommand{\co}{{\rm co}}

\def\ees{{\accent"5E e}\kern-.385em\raise.2ex\hbox{\char'23}\kern-.08em}
\def\EES{{\accent"5E e}\kern-.5em\raise.8ex\hbox{\char'23 }}
\def\ow{o\kern-.42em\raise.82ex\hbox{\vrule width .12em height .0ex depth .075ex \kern-0.16em \char'56}\kern-.07em}
\def\OW{o\kern-.460em\raise1.36ex\hbox{
\vrule width .13em height .0ex depth .075ex \kern-0.16em
\char'56}\kern-.07em}
\def\DD{D\kern-.7em\raise0.4ex\hbox{\char '55}\kern.33em}

\title{The mountain pass theorem in terms of tangencies}

\author{S\~i Ti\d{\^e}p \DD inh$^\dagger$}
\address{Institute of Mathematics, VAST, 18, Hoang Quoc Viet Road, Cau Giay District 10307, Hanoi, Vietnam}
\email{dstiep@math.ac.vn}

\author{Ti\'{\^e}n-S\OW n Ph\d{a}m$^\ddagger$}
\address{Department of Mathematics, Dalat University, 1 Phu Dong Thien Vuong, Dalat, Vietnam}
\email{sonpt@dlu.edu.vn}



\date{ \today}

\begin{document}

\begin{abstract} 
This paper addresses the Mountain Pass Theorem for locally Lipschitz functions on finite-dimensional vector spaces in terms of tangencies. Namely, let $f \colon \mathbb R^n \to \mathbb R$ be a locally Lipschitz function with a mountain pass geometry.
Let $$c := \inf_{\gamma \in \mathcal A}\max_{t\in[0,1]}f(\gamma(t)),$$ where $\mathcal{A}$ is the set of all continuous paths joining $x^*$ to $y^*.$ We show that either $c$ is a critical value of $f$ or $c$ is a tangency value at infinity of $f.$ This reduces to the Mountain Pass Theorem of Ambrosetti and Rabinowitz in the case where the function $f$ is definable (such as, semi-algebraic) in an o-minimal structure.
\end{abstract}

\maketitle

\pagestyle{plain}

\section{Introduction}
The celebrated Mountain Pass Theorem of Ambrosetti and Rabinowitz \cite{Ambrosetti1973} is a very useful tool in nonlinear analysis with many important applications. For more details, we refer the reader to the comprehensive monographs \cite{Ambrosetti2007, Jabri2003, Nirenberg1981, Rabinowitz1986, Struwe1996, Willem1996} with the references therein.

The aim of this paper is to provide a version of the Mountain Pass Theorem for locally Lipschitz functions on finite-dimensional vector spaces in terms of tangencies. To be more precise, let us recall some basic terminology.

Let $f \colon \mathbb R^n \to \mathbb R$ be a $C^1$ function and let $x^*, y^* \in \mathbb{R}^n$ be such that there exists an open neighborhood $\mathcal{U}$ of $x^*$ satisfying the conditions $y^* \not \in \overline{\mathcal{U}}$ and 
$$\max\{f(x^*), f(y^*)\}  < \inf_{x \in \partial \mathcal{U}} f(x).$$
Consider the family  $\mathcal{A}$ of all continuous paths joining $x^*$ to $y^*$ and set
\begin{equation*}
c := \inf_{\gamma \in \mathcal A}\max_{t\in[0,1]}f(\gamma(t)). 
\end{equation*}
The following relation is well-known and has many interesting applications:
$$c \in K_0(f) \cup K_\infty(f),$$
where $K_0(f)$ is the set of {\em critical values} of $f$ and $K_\infty(f)$ is the set of values at which $f$ does not satisfy the {\em weak Palais--Smale condition}, i.e.,
$$K_0(f) := \left\{
\begin{array}{lll}
t\in\mathbb R:& \text{there is a point } x \in \mathbb{R}^n \text { such that } \nabla f(x) = 0 \text{ and }  f(x) = t
\end{array}\right\}$$
and 
$$K_\infty(f) := \left\{
\begin{array}{lll}
t\in\mathbb R:& \text{there is a sequence } x^k\to\infty \text { such that }\\
&f(x^k)\to t \text{ and } \|x^k\| \|\nabla f(x^k)\|  \to 0
\end{array}\right\}.
$$

\medskip
In a difference line of development, suppose that the function $f$ is polynomial (or more general, definable in an o-minimal structure; see \cite{Dries1998} for more on the subject).
It is well known that there exists a (minimal) finite set $B(f) \subset  \mathbb{R},$ called the {\em bifurcation set} of $f,$ such that the restriction map
\begin{eqnarray*}
f \colon \mathbb{R}^n \setminus f^{-1}(B(f)) \rightarrow \mathbb{R} \setminus B(f)
\end{eqnarray*}
is a locally trivial $C^\infty$-fibration (see, for example, \cite{Broughton1988, HaHV1984, HaHV2008-1, Parusinski1995, Rabier1997, Thom1969, Siersma1995}). 
Since $f$ may not be proper, the bifurcation set $B(f)$ contains not only the set of critical values $K_0(f),$ but also 
the set $B_\infty(f)$ of {\em atypical values at infinity} corresponding to the {\em critical points at infinity}.
While the set $K_0(f)$ is relatively well understood, the other set $B_\infty(f)$ is still mysterious. To control the set $B_\infty(f),$ we can use the set $T_\infty(f)$ of {\em tangency values (at infinity)} of $f$:
$$T_\infty(f):=\left\{
\begin{array}{lll}
t\in\mathbb R:& \text{there is a sequence } x^k\to\infty \textrm{ such that }\\
&f(x^k)\to t \text{ and } \rank\{\nabla f(x^k), x^k\}=1
\end{array}\right\}.$$
It is well-known (see, for example, \cite{Acunto2000, HaHV2008-2, HaHV2017, Kim2019, TaLL1998}) that $T_\infty(f)$ is a finite set and 
$$B_\infty(f) \subset T_\infty(f) \subset K_\infty(f).$$
The inclusions may be strict (see \cite{Paunescu1997, Paunescu2000}).

\medskip
Motivated by the aforementioned works and the usefulness of tangencies in semi-algebraic optimization (see \cite{HaHV2008-2, HaHV2009-1, HaHV2017, Kim2019, PHAMTS2020} for more details), we will show in Theorem~\ref{MountainPass} that
$$c \in K_0(f) \cup T_\infty(f).$$
Actually, the same conclusion holds even when the function $f$ is not assumed to be differentiable, but merely locally Lipschitz continuous.

\medskip
The rest of the paper is organized as follows. Section~\ref{PreliminariesSection} covers some preliminary materials. Section~\ref{MainResultSection}  presents the main result and its proof.

\section{Preliminaries} \label{PreliminariesSection}

\subsection{Notation}
Throughout this work we shall consider the Euclidean vector space ${\Bbb R}^n$ endowed with its canonical scalar product $\langle \cdot, \cdot \rangle,$ and we shall denote its associated norm $\| \cdot \|.$ The open ball (resp., the sphere) centered at $\bar{x} \in \mathbb{R}^n$ of radius $r$ will be denoted by $\mathbb{B}_{r}(\bar{x})$ (resp., $\mathbb{S}_{r}(\bar{x})$). For simplicity, we write $\mathbb B^n_r$ and $\mathbb S^{n-1}_r$ if $x=0$; and write $\mathbb B^n$ and $\mathbb S^{n-1}$ if $x=0$ and $r=1$. For a subset $A$ of $\mathbb R^n,$ the closure, the boundary and the convex hull of $A$ are denoted by $\overline A$, $\partial A$ and $\co(A)$ respectively.
Let $\dist(A, B)$ stand for the Euclidean distance between $A$ and $B\subset\mathbb R^n$, namely
$$\dist(A, B):=\inf\{\|x-y\|:\ x\in A,\ y\in B\}.$$
For convenience, if $B\ne\emptyset$, set $\dist(\emptyset, B):=+\infty.$ 

\subsection{Subdifferential of locally Lipschitz mappings}

Here we recall the notions and some elementary properties of the Clarke subdifferential and the generalized directional derivative of locally Lipschitz functions used in this paper. The reader is referred to \cite{Clarke1976, Clarke1981, Clarke1990} for more details.

\begin{definition}
Let $F\colon\mathbb R^n\to\mathbb R^m$ be a locally Lipschitz mapping. The {\em Clarke subdifferential} of $F$ at $x\in\mathbb R^n$ is defined by
$$\partial F(x) := \co\{\lim d_{x^k}F:\ x^k\to x\ \text{ and }\ F \ \text {is differentiable at}\ x^k\},$$
where $d_{x^k}F$ is the differential of $F$ at $x^k$, which can be identified with the Jacobian matrix of $F$ at $x^k.$
\end{definition}

\begin{definition}
Let $f\colon\mathbb R^n\to\mathbb R$ be a locally Lipschitz function and $v\in\mathbb R^n$. 
The {\em generalized directional derivative} of $f$ at $x$ in the direction $v$, denoted by $f^\circ(x;v),$
is defined as follows:
$$f^\circ(x;v) := \limsup_{{y\to x, h\to 0^+}}\frac{f(y+hv)-f(y)}{h}.$$
\end{definition}

\begin{lemma}\label{Clarke} \cite[Propositions~2.1.2~and~2.1.5]{Clarke1990} Let $f\colon\mathbb R^n\to\mathbb R$ be a locally Lipschitz function. Then we have: 
\begin{enumerate}[{\rm (i)}]
\item For all $x\in\mathbb R^n$, the set $\partial f(x)$ is a non-empty, convex, compact subset of $\mathbb R^n.$ 
\item The set-valued mapping $\partial f$ is upper semi-continuous on $\mathbb R^n$, i.e., for any $x\in\mathbb R^n$, if $x^k\in\mathbb R^n$ and $w^k\in\partial f(x^k)$ are sequences such that $x^k\to x$ and $w^k\to w,$ then $w\in\partial f(x).$
\item $f^\circ(x;v)=\max_{w\in\partial f(x)}\langle w,v \rangle$ for any $v\in\mathbb R^n$.
\end{enumerate}
\end{lemma}

The following lemma is a slightly changed version of~\cite[Lemma~3.3]{Chang1981}.
\begin{lemma}\label{Chang} Let $f\colon\mathbb R^n\to\mathbb R$ be a locally Lipschitz function and $b>0.$
Suppose that $U\subset\mathbb R^n$ is an open set such that 
$$\inf_{w\in\partial f(x)}\|w\|\geqslant 2b \ \text{ for all }\ x\in U.$$
Then there exists a locally Lipschitz vector field $\vv(x)$ defined on $U$ satisfying 
$$\|\vv(x)\|<1\ \text{ and } \ \langle w,\vv(x) \rangle>b \text{ for any } w\in\partial f(x).$$
\end{lemma}
\begin{proof} 
Note that the assumption of the lemma is the conclusion of~\cite[Lemma~3.2]{Chang1981} which is used to prove~\cite[Lemma~3.3]{Chang1981}, so the proof is completely similar to that of~\cite[Lemma~3.3]{Chang1981}.
\end{proof}

In the sequel, we will need the following lemma.
\begin{lemma}\label{direction} Let $f\colon\mathbb R^n\to\mathbb R$ be a locally Lipschitz function and $D\subset\mathbb R^n$ be a compact set. For each $\epsilon>0$,  there is $h_0=h_0(\epsilon)>0$ such that for all $y\in N_{h_0}(D)$, $h\in (0,h_0]$ and $v\in\overline{\mathbb B}^n$, we have 
$$\frac{f(y+hv)-f(y)}{h}<f^\circ(x;v)+\epsilon,$$
where $N_{h_0}(D)$ is the closed neighborhood of radius $h_0$ of $D$ and $x\in D$ is a point such that $\dist(y,D)=\|y-x\|.$
\end{lemma}
\begin{proof} 
By contradiction, assume that there is $\epsilon_0>0$ such that for all integer $k>0$, there are $y^k\in N_{\frac{1}{k}}(D),$ $x^k\in D$, $h_k\in (0,\frac{1}{k}]$ and $v^k\in\overline{\mathbb B}^n$ such that 
\begin{equation}\label{yv}\dist(y^k,D)=\|y^k-x^k\|\ \ \text{ and }\ \ \frac{f(y^k+h_k v^k)-f(y^k)}{h_k}\geqslant f^\circ(x^k;v^k)+\epsilon.\end{equation}

By taking subsequences if necessary, we can suppose that the sequences $x^k$ and $v^k$ converge to the limits $x^0$ and $v^0,$ respectively. 
Since $f$ is locally Lipschitz, there is a neighborhood $U$ of $x^0$ and a constant $K>0$ such that $f$ is Lipschitz on $U$ with the constant $K.$ 
In addition, by definition and by shrinking $U$ if necessary, we can suppose that there is $\widetilde h>0$ such that
$$\frac{f(y+hv^0)-f(y)}{h}<f^\circ(x^0;v^0)+\frac{\epsilon}{4}$$
for all $y\in U$ and $h\in(0,\widetilde h].$
Set 
$$z^k:=y^k+h_k(v^k-v^0).$$
It is clear that $y^k\to x^0$, $z^k\to x^0$ and $z^k+h_kv^0\to x^0$ as $k\to+\infty$. 
Consequently, for $k$ large enough so that $y^k,z^k\in U$, $\displaystyle K\|v^k-v^0\|<\frac{\epsilon}{4}$ and $h_k<\widetilde h$, we have 
$$\begin{array}{llll}
\displaystyle\frac{f(y^k+h_k v^k)-f(y^k)}{h_k}
&=&\displaystyle\frac{f(y^k+h_k(v^k-v^0)+h_k v^0)-f(y^k+h_k(v^k-v^0))}{h_k}\\
&&+\displaystyle\frac{f(y^k+h_k(v^k-v^0))-f(y^k)}{h_k}\\
&=&\displaystyle\frac{f(z^k+h_k v^0)-f(z^k)}{h_k}+\frac{f(z^k)-f(y^k)}{h_k}\\
&<&\displaystyle f^\circ(x^0;v^0)+\frac{\epsilon}{4}+K\|v^k-v^0\|<f^\circ(x^0;v^0)+\frac{\epsilon}{2}.
\end{array}$$
This contradicts~\eqref{yv} and so ends the proof of the lemma.
\end{proof}

\section{The main result and its proof} \label{MainResultSection}

For a locally Lipschitz function $f\colon\mathbb R^n\to\mathbb R,$ we define the set of {\em critical values} of $f$ and the set of {\em tangency values (at infinity)} of $f,$ respectively, by
$$K_0(f):=\{t\in\mathbb R:\ \text{there is } x\in f^{-1}(t) \text{ such that } 0 \in \partial f(x)\}$$
and
$$T_\infty(f):=\left\{
\begin{array}{lll}
t\in\mathbb R:& \text{there are sequences } x^k\to\infty \text{ and } v^k\in\partial f(x^k) \text { such that }\\
&f(x^k)\to t \text{ and } \rank\{x^k,v^k\}=1
\end{array}\right\}.
$$

The main result of the paper is as follows. 
\begin{theorem}[Mountain pass]  \label{MountainPass}
Let $f\colon\mathbb R^n\to\mathbb R$ be a locally Lipschitz function and let $x^*,y^*\in\mathbb R^n$ with $x^*\ne y^*.$ Assume that there is an open neighborhood $\mathcal U$ of $x^*$ such that $y^*\not\in\overline{\mathcal  U}$ and 
$$f(x^*),f(y^*)<\inf_{x\in\partial \mathcal U} f(x).$$
Let 
\begin{eqnarray}\label{c}
c:=\inf_{\gamma\in\mathcal A}\max_{t\in[0,1]}f(\gamma(t)),
\end{eqnarray}
where $\mathcal A$ stands for the set of all continuous path joining $x^*$ to $y^*,$ i.e.,
\begin{eqnarray}\label{A}
\mathcal A &:=& \{\gamma\in C([0,1],\mathbb R^n):\ \gamma(0)=x^*,\ \gamma(1)=y^*\}.
\end{eqnarray}
Then 
$$c\in K_0(f)\cup T_\infty(f).$$
\end{theorem}

Let us start with some lemmas of preparation. 


\begin{lemma}\label{Cover} Let $X\subset\mathbb R^n$ be a compact set and let $\mathcal Z:=\{Z_i:\ i=1,\dots,p\}$ be a distinct finite open cover of $X$. Then there exists a constant $\lambda>0$ depending on $\mathcal Z$ such that the following statements hold:
\begin{enumerate}[{\rm (i)}]
\item For any $i\in\{1,\dots,p\}$ and any $x\in Z_i\cap X$ such that $Z_i$ is the unique open set in the cover containing $x,$ then $\dist(x,\partial Z_i)>3\lambda.$
\item For any $i\in\{1,\dots,p\}$ and any $x\in Z_i\cap X$ such that $\dist(x,\partial Z_i)\leqslant 3\lambda$, there is $j\in\{1,\dots,p\}\setminus\{i\}$ depending on $x$ such that
$$x\in Z_j\ \text{ and }\ \dist(x,\partial Z_j)>3\lambda.$$
\item Let $i\in\{1,\dots,p\}$ and $x\in Z_i\cap X$ be such that $\dist(x,\partial Z_i)={2\lambda}$ and $\dist(x,X)\leqslant{\lambda}$. 
Then for each $y\in X$ with $\|x-y\|\leqslant\lambda$, we have $\dist(y,\partial Z_i)\geqslant{\lambda}$ and there exists $j\in\{1,\dots,p\}\setminus\{i\}$ such that 
$$x,y\in Z_j,\ \dist(x,\partial Z_j)>{2\lambda}\ \text{ and }\ \dist(y,\partial Z_j)>3{\lambda}.$$
\end{enumerate}
\end{lemma}
\begin{proof} 
(i) By contradiction, assume that for any integer $k>0$, there is an index $i_k\in\{1,\dots,p\}$ and a point $x^k\in Z_{i_k}\cap X$ such that $Z_{i_k}$ is the unique open set in $\mathcal Z$ containing $x^k$ and $\dist(x^k,\partial Z_{i_k})\leqslant\frac{3}{k}$. 
Since $\mathcal Z$ is finite, by taking a subsequence if necessary, we can suppose that $i_k$ is fixed for all $k$, namely, $i_k=i\in\{1,\dots,p\}.$ 
By the compactness of $X$, we can assume that the sequence $x^k$ converges to a limit $x^0\in X$. 
Clearly $x^0\in\partial Z_{i}\cap X,$ in particular, $x^0\not\in Z_i$.
Furthermore, $x^0\not\in Z_j$ for all $j\ne i$ since otherwise, $x^k\in Z_j$ for $k$ large enough which contradicts the fact that $Z_i$ is the unique open set in $\mathcal Z$ containing $x^k$. 
Consequently, $\displaystyle x^0\notin\cup_{j=1,\dots,p}Z_j$. So $x^0\notin X$, which is a contradiction.

(ii) Let $\lambda$ be given by item (i). We will show that the statement holds by shrinking $\lambda.$ 
For contradiction, assume that for any integer $k>0$ such that $\frac{1}{k}<\lambda$, there is an index $i_k\in\{1,\dots,p\}$ and a point $x^k\in Z_{i_k}\cap X$ such that 
\begin{itemize}
\item $\dist(x^k,\partial Z_{i_k})\leqslant \frac{3}{k}$; and 
\item for each $j\in\{1,\dots,p\}\setminus\{i\}$, either $x^k\notin Z_{j}$ or $x^k\in Z_j$ with $\dist(x^k,\partial Z_j)\leqslant\frac{3}{k}$. 
\end{itemize}
Since the cover is finite, by taking a subsequence if necessary, we may suppose that $i_k$ is fixed for all $k$, namely, $i_k=i\in\{1,\dots,p\};$ 
moreover, for each $j\in\{1,\dots,p\}\setminus\{i\}$, either $x^k\notin Z_{j}$ for all $k$ or $x^k\in Z_j$ and $\dist(x^k,\partial Z_j)\leqslant\frac{3}{k}$ for all $k$.
As $X$ is compact, we may assume that the sequence $x^k$ converges to a limit $x^0\in X$. 
Clearly $x^0\in\partial Z_{i}\cap X.$ 
For each $j\in\{1,\dots,p\}\setminus\{i\}$, by construction, either $x^0\notin\overline Z_j$ or $x^0\in\overline Z_j$ with $\dist(x^0,\partial Z_j)=0$, in particular, $x^0\not\in Z_j$. 
Therefore $\displaystyle x^0\notin\cup_{j=1,\dots,p}Z_j$. Consequently $x^0\notin X$, which is a contradiction.

(iii) Let $\lambda$ be given by (i) and (ii). 
By assumption, $\mathbb B^n_{2\lambda}(x)\subset Z_i.$
Let $y\in X$ be such that $\|x-y\|\leqslant\lambda$. 
Then 
$y\in Z_i$. 
On the other hand,
$$\dist(y,\partial Z_i)\leqslant\|x-y\|+\dist(x,\partial Z_i)\leqslant 3\lambda.$$
By item (ii), there is $j\in\{1,\dots,p\}\setminus\{i\}$ such that $y\in Z_j$ and 
$$\dist(y,\partial Z_j)>3\lambda.$$
Therefore $\mathbb B^n_{3\lambda}(y)\subset Z_j$ and so $x\in Z_j.$
In addition,
$$\dist(x,\partial Z_j)\geqslant\dist(y,\partial Z_j)-\|x-y\|>2\lambda.$$
This ends the proof of the lemma.
\end{proof}

\begin{lemma}\label{Decompose} Let $X\subset\mathbb R^n$ be a compact set and let $\mathcal Z:=\{Z_i:\ i=1,\dots,p\}$ be a distinct finite family of open balls covering  $X$. 
Assume that, for each $i$, there is a constant $\rho_i>0$, an open set $W_i\supset Z_i$ and a bi-Lipschitz homeomorphism $\eta_i\colon W_i\to\eta_i(W_i)\subset\mathbb R^n$ such that 
$$\eta_i(W_i)=\{u\in\mathbb R^n:\ |u_j|<\rho_i,\ j=1,\dots,n\}\ \text{ and }\ \eta_i(W_i\cap X)\subset\{u\in\mathbb R^n:\ u_2=0\}.$$
Let $\lambda>0$ be the constant depending on $\mathcal Z$ given by Lemma~\ref{Cover} and $L\geqslant 1$ be a common Lipschitz constant of $\eta_1,\dots,\eta_p,\eta_1^{-1},\dots,\eta_p^{-1}.$ 
Let $\gamma\colon[a,b]\to\mathbb R^n$ be a continuous piecewise linear curve such that $\dist(\gamma(t),X)\leqslant\frac{\lambda}{2L^2}$ for $t\in[a,b].$
Then there exist finite sequences $a=:T_0<\cdots<T_q:=b$ and $i_0\ne i_1\ne\cdots\ne i_{q-1}$ with $i_k\in\{1,\dots,p\}\ (k=0,\dots,q-1)$ such that:
\begin{enumerate}[{\rm (i)}]
\item $\gamma[T_{k-1},T_{k}]\subset W_{i_{k-1}}$ for $k=1,\dots,q$; and
\item $\eta_{i_k}^{-1}(w^k)\in W_{i_{k-1}}\cap W_{i_{k}}$ for $k=1,\dots,q-1,$ where 
$$z^k:=\eta_{i_k}(\gamma(T_k))\ \text{ and }\ w^k:=(z^k_1,w^k_2,z^k_3,\dots,z^k_n)$$
with
$$w_2^k=\left\{\begin{array}{llll}
-z_2^k             &\text{if}& z_2^k\ne 0\\
-\frac{\lambda}{L} &\text{if}& z_2^k=   0.
\end{array}\right.$$
\end{enumerate}
\end{lemma}
\begin{proof}
The construction of the desired sequence is done by induction as follow.

\subsubsection*{Step 1: $k=0$.} Let $y^0\in X$ be such that $\|y^0-\gamma(T_0)\|=\dist(\gamma(T_0),X)<\lambda.$ 
In view of Lemma~\ref{Cover}(i)-(ii), there is $i_0\in\{1,\dots,p\}$ such that $y^0\in Z_{i_0}$ and $\dist(y^0,\partial Z_{i_0})>3\lambda.$ 
So 
$$\dist(\gamma(T_0),\partial Z_{i_0})\geqslant\dist(y^0,\partial Z_{i_0})-\|y^0-\gamma(T_0)\|>2\lambda.$$
Hence $\gamma(T_0)\in Z_{i_0}.$
Set $$S_1:=\sup\{t\in [T_0,b]:\ \gamma(s)\in Z_{i_0}\ \text{for all}\ s\in[T_0,t]\}>T_0.$$
If $S_1=b$, then set $T_1:=b$ and we are done. 
Otherwise, we have $\gamma(S_1)\in\partial Z_{i_0}$. 
So if we let
$$T_1:=\sup\{t\in [T_0,b]:\ \dist(\gamma(s),\partial Z_{i_0})\geqslant 2\lambda\ \text{for all}\ s\in[T_0,t]\},$$
then clearly  
$$T_0<T_1<b,\ \ \gamma[T_{0},T_{1}]\subset W_{i_{0}},\ \text{ and }\ \dist(\gamma(T_1),\partial Z_{i_0})= 2\lambda.$$ 
By Lemma~\ref{Cover}(iii), there is $i_1\ne i_0$ such that $\gamma(T_1)\in Z_{i_1}$ and $\dist(\gamma(T_1),\partial Z_{i_1})> 2\lambda.$ 

\subsubsection*{Step 2: Induction.} For $k\geqslant 0$, assume that we have constructed sequences $a=:T_0<\dots<T_k$, $i_0\ne i_1\ne\cdots\ne i_{k}$ and $\{y^0,\dots,y^{k}\}$ such that, for $l=1,\dots,k$, we have:
\begin{enumerate}[{\rm (a)}]
\item $T_{l}:=\sup\{t\in [T_{l-1},b]:\ \dist(\gamma(s),\partial Z_{i_{l-1}})\geqslant 2\lambda$ for all $s\in[T_{l-1},t]\}$;
\item $\dist(\gamma(T_{l}),\partial Z_{i_{l-1}})= 2\lambda $ and $\dist(\gamma(T_{l}),\partial Z_{i_{l}})> 2\lambda$;
\item $\gamma[T_{l-1},T_{l}]\subset W_{i_{l-1}}$; 
\item $\|y^{l-1}-\gamma(T_{l-1})\|=\dist(\gamma(T_{l-1}),X)$ and  $\|y^k-\gamma(T_k)\|=\dist(\gamma(T_k),X)$.
\end{enumerate}

Set $$S_{k+1}:=\sup\{t\in [T_k,b]:\ \gamma(s)\in Z_{i_k}\ \text{for all}\ s\in[T_k,t]\}>T_k.$$
If $S_{k+1}=1$, set $q:=k+1$ and $T_{k+1}:=b$. 
Then item (i) follows.
Contrarily, we have $\gamma(S_{k+1})\in\partial Z_{i_k}$.
Let
$$T_{k+1}:=\sup\{t\in [T_k,b]:\ \dist(\gamma(s),\partial Z_{i_k})\leqslant 2\lambda\ \text{for all}\ s\in[T_k,t]\},$$
then clearly $T_k<T_{k+1}<b$ and $\dist(\gamma(T_{k+1}),\partial Z_{i_k})= 2\lambda$. 
By Lemma~\ref{Cover}(iii), there is $i_{k+1}\ne i_k$ such that $\gamma(T_{k+1})\in Z_{i_{k+1}}$ and $\dist(\gamma(T_{k+1}),\partial Z_{i_{k+1}})> 2\lambda.$ 
Hence the process can be repeated with $k$ replaced by $k+1.$

Observe that the sequence $a=:T_0<T_1<\cdots$ is finite so there must be $q>0$ such that $T_q=b.$ 
Indeed, suppose for contradiction that the sequence $a=:T_0<T_1<\cdots$ is infinite. 
Then there exists an index $k$ that appears infinitely many times in the sequence $i_0\ne i_1\ne\cdots$.
This implies that, in view of item (b), $\gamma$ cuts the sphere $\{x\in Z_k:\ \dist(z,\partial Z_k)=2\lambda\}$ infinitely many times which is a contradiction since a piecewise linear curve can meets a sphere at finitely many times. 
Consequently, (i) follows immediately.

Now we show that $\eta_{i_k}^{-1}(w^k)\in W_{i_{k-1}}\cap W_{i_{k}}$ for $k=1,\dots,q-1.$
If $z_2^k=0,$ since $\eta^{-1}_{i_k}$ is Lipschitz with the constant $L$, then 
\begin{equation}\label{zw1}\|\eta_{i_k}^{-1}(z^k)-\eta_{i_k}^{-1}(w^k)\|\leqslant L\|z^k-w^k\|=L\|z_2^k-w_2^k\|=\lambda.\end{equation}
Otherwise, by item (d) and by the assumption $\dist(\gamma(t),X)\leqslant\frac{\lambda}{2L^2}$ for $t\in[a,b],$ we have $\|y^k-\gamma(T_k)\|<\frac{\lambda}{2L^2}.$
Hence $\|\eta_{i_k}(y^k)-z^k\|<\frac{\lambda}{2L},$ and so
\begin{equation}\label{zw2}\begin{array}{lll}
\|\eta_{i_k}^{-1}(z^k)-\eta_{i_k}^{-1}(w^k)\|&\leqslant&\|\eta_{i_k}^{-1}(z^k)-\eta_{i_k}^{-1}(\eta_{i_k}(y^k))\|+\|\eta_{i_k}^{-1}(\eta_{i_k}(y^k))-\eta_{i_k}^{-1}(w^k)\|\\
&\leqslant&\|y^k-\gamma(T_k)\|+L\|\eta_{i_k}(y^k)-w^k\|\\
&=&\|y^k-\gamma(T_k)\|+L\|\eta_{i_k}(y^k)-z^k\|\\
&\leqslant&(1+L^2)\|y^k-\gamma(T_k)\|
<\displaystyle (1+L^2)\frac{\lambda}{2L^2}\leqslant\lambda.
\end{array}\end{equation}
Now~\eqref{zw1} and~\eqref{zw2}, together with the facts $\dist(\gamma(T_{k}),\partial Z_{i_{k-1}})= 2\lambda$, give 
$$\begin{array}{lll}
\dist(\eta_{i_k}^{-1}(w^k),\partial Z_{i_{k-1}})&\geqslant&\dist(\eta_{i_k}^{-1}(z^k),\partial Z_{i_{k-1}})-\|\eta_{i_k}^{-1}(z^k)-\eta_{i_k}^{-1}(w^k)\|\\
&=&\dist(\gamma(T_k),\partial Z_{i_{k-1}})-\|\eta_{i_k}^{-1}(z^k)-\eta_{i_k}^{-1}(w^k)\|>\lambda.
\end{array}$$
Similarly,~\eqref{zw1},~\eqref{zw2} and $\dist(\gamma(T_k),\partial Z_{i_k})> 2\lambda$ imply $\dist(\eta_{i_k}^{-1}(w^k),\partial Z_{i_{k}})>\lambda.$ 
Therefore $\eta_{i_k}^{-1}(w^k)\in W_{i_{k-1}}\cap W_{i_{k}}$ for $k=1,\dots,q-1.$ 
So~(ii) follows and the lemma is proved.
\end{proof}

We need the following variant of the constant rank theorem for locally Lipschitz mappings.
\begin{lemma}\label{CR} Let $F\colon\mathbb R^n\to\mathbb R^m$ be a locally Lipschitz mapping with $n\geqslant m$. 
Assume that each element of $\partial F(x)$ has rank $m$ for any $x$ in a neighborhood of $x^0\in\mathbb R^n$. 
Then there is an open neighborhood $Z$ of $x^0$ in $\mathbb R^n$ and a bi-Lipschitz homeomorphism $\eta\colon Z\to\eta(Z)\subset\mathbb R^n$ such that 
$$F\circ\eta^{-1}(u_1,\dots,u_n)=(u_1,\dots,u_m)+F(x^0),$$
for all $(u_1,\dots,u_n)\in\eta(Z).$
\end{lemma}
\begin{proof} The proof follows directly from the proof of~\cite[Theorem 3.1]{Butler1988}.
\end{proof}

\begin{proof} [Proof of Theorem~\ref{MountainPass}]

Recall that $\mathcal A$ is the set given by~\eqref{A}. 
For $r\geqslant\max\{\|x^*\|,\|y^*\|\}$ and $\epsilon>0,$ set 
\begin{equation}\label{Are}\mathcal A(r,\epsilon):=\left\{\gamma\in \mathcal A:\ \max_{t\in[0,1]}f(\gamma(t))< c+\epsilon\ \text{ and } \ \max_{t\in[0,1]}\|\gamma(t)\|\leqslant r\right\}.\end{equation}
By definition, for each $\epsilon>0$, there exists $r\gg 1$ such that $\mathcal A(r,\epsilon)$ is non-empty so the function
\begin{equation}\label{Re}
(0,+\infty)\to(0,+\infty),\ \ \epsilon\mapsto R(\epsilon):=\inf\{r:\ \mathcal A(r,\epsilon)\ne\emptyset\}
\end{equation}
is well-defined and moreover, it is decreasing. 
In particular, there exists the limit
$$\lim_{\epsilon\to 0^+} R(\epsilon) \in \mathbb{R} \cup \{+\infty\}.$$ 
Now Theorem~\ref{MountainPass} is a direct consequence of Propositions~\ref{Prop1} and~\ref{Prop2} below. 
\end{proof}

\begin{proposition}\label{Prop1} 
If $\lim_{\epsilon\to 0^+} R(\epsilon)  < +\infty$ then $c\in K_0(f)$.
\end{proposition}
We need some preparation before giving the proof of Proposition~\ref{Prop1}.
As $R(\epsilon)$ is decreasing, there is a constant $R_0>0$ such that for all $\epsilon>0$, we have $R(\epsilon)<R_0$ and so $\mathcal A(R_0,\epsilon)\ne\emptyset$. 
For each integer $k>0$, take $\gamma^k\in\mathcal A\left(R_0,\frac{1}{k}\right)$ and let $D$ be the superior Kuratowski limit of the sequence of non-empty compact sets 
$$D_k:=\{\gamma^k(t):\ t\in[0,1] \text{ and } f(\gamma^k(t))\geqslant c\}.$$
Namely, $x\in D$ if and only if there is a sequence $x^{k_l}\in D_{k_l}$ such that $x^{k_l}\to x$ as $l \to +\infty.$
It is clear that $D$ is a non-empty compact set and $f(x)=c$ for any $x\in D.$
To prove that $c\in K_0(f)$, it is enough to show that there is $x\in D$ such that $0\in\partial f(x).$
Assume for contradiction that $0\not\in\partial f(x)$ for all $x\in D.$
By the compactness of the Clarke subdifferential (Lemma~\ref{Clarke}(i)), it is not hard to see that for each $x\in D,$ there is a constant $b_x>0$ such that 
$$\inf_{w\in\partial f(x)}\|w\|>4b_x.$$ 
By Lemma~\ref{Clarke}(ii), there exists a bounded open neighborhood $U_x$ of $x$ such that 
$$\inf_{w\in\partial f(y)}\|w\|>2b_x\ \text{ for all }\ y\in U_x.$$
By assumption, we have 
$$f(x^*),f(y^*)<\inf_{x\in\partial \mathcal U} f(x)\leqslant c.$$
Thus $x^*,y^*\not\in D$ and so, we can shrink $U_x$ so that $x^*,y^*\not\in U_x.$

As $D$ is compact and $\{U_x:\ x\in D\}$ is an open cover of $D$, there exists a finite open cover of $D$: 
$$\{U_{x^i}:\ x^i\in D,\ i=1,\dots,p\}.$$ 
Let 
$$\displaystyle b:=\min_{i=1,\dots,p}b_{x^i}>0\ \text{ and }\ \displaystyle U:=\bigcup_{i=1}^p U_{x^i}.$$
In view of Lemma~\ref{Chang}, there exists a locally Lipschitz vector field $\vv(x)$ defined on $U$ such that 
\begin{equation}\label{v}\|\vv(x)\|<1\   \text{ and }\  \langle w,\vv(x) \rangle>b\ \text{for any}\ w\in\partial f(x).\end{equation}
Let $h_0:=h_0\left(\frac{b}{4}\right)$ be the constant determined by Lemma~\ref{direction}. So, for all $y\in N_{h_0}(D)$, $h\in (0,h_0]$ and $v\in\overline{\mathbb B}^n$, we have 
\begin{equation}\label{b/4}\frac{f(y+hv)-f(y)}{h}<f^\circ(x;v)+\frac{b}{4},\end{equation}
where $x\in D$ is a point such that $\dist(y,D)=\|y-x\|.$
Let $V\subset U$ be an open neighborhood of $D$ such that $\overline{V}\subset U\cap N_{h_0}(D).$ 
By a smooth version of Urysohn's lemma~\cite[Lemma~1.3.2]{Petersen}, there is a smooth function $\varphi\colon\mathbb R^n\to [0,1]$ such that:
\begin{equation}\label{varphi}
\varphi(\mathbb R^n\setminus U)=0\  \text{ and }\  \varphi(\overline{V})=1.
\end{equation}
Let 
$$\widetilde \vv(x):=\varphi(x)\vv(x),$$
which is obviously a locally Lipschitz vector field on $\mathbb R^n$. 
This, together with the facts that $U$ is bounded and $\supp(\widetilde \vv)\subset U$, implies that the vector field $\widetilde \vv$ is Lipschitz on $\mathbb R^n$ with a constant $K>0$.
Moreover, in view of Lemma~\ref{Clarke}(i),~(ii), there is a constant $K'>0$ such that $\|w\|\leqslant K'$ for all $w\in \partial f(x)$ and $x\in \overline U.$
We need the following lemma.

\begin{lemma}\label{decreasing}
For any trajectory $\alpha\colon (t_1,t_2)\to \mathbb R^n$ of $-\widetilde \vv$, the function $f\circ\alpha$ is decreasing on $(t_1,t_2)$. 
In addition, for all $u\in (t_1,t_2)$ such that $\alpha(u)\in V$, we have
\begin{equation}\label{novarphi}f(\alpha(u+h))-f(\alpha(u))<-\frac{bh}{2} \ \text{ for }\ h\in\left(0,\min\left\{\frac{b}{2KK'},h_0\right\}\right],\end{equation}
where $h_0:	=h_0\left(\frac{b}{4}\right)$ is the constant determined by Lemma~\ref{direction}.
\end{lemma}
\begin{proof}
Take any $u\in (t_1,t_2)$. For all $s\in(t_1,t_2)$, we have
\begin{equation}\label{su}\displaystyle\|\alpha(s)-\alpha(u)\|=\left\|\int_u^{s} -\widetilde \vv(\alpha(s'))ds'\right\|\leqslant\left|\int_u^{s} \|\widetilde \vv(\alpha(s'))\|ds'\right|\leqslant\left| \int_u^{s}ds'\right|=|s-u|.\end{equation}
Observe that the first statement is clear if $\alpha(u)\not\in\supp~\widetilde \vv\subset U$ so assume that $\alpha(u)\in \supp~\widetilde \vv.$
Thus $\varphi(\alpha(u))>0$ and so $\alpha(u)\in U$. 
Let 
$$\epsilon:=\frac{\varphi(\alpha(u))}{4}b>0.$$ 
For all $h\in\Big(0,\frac{\varphi(\alpha(u))b}{2KK'}\Big]$ small enough, we have
\begin{equation}\label{long}\begin{array}{lll}
\displaystyle\frac{f(\alpha(u+h))-f(\alpha(u))}{h}&<&\displaystyle f^\circ\left[\alpha(u);\frac{\alpha(u+h)-\alpha(u)}{h}\right]+\epsilon\\
&=&\displaystyle\max_{w\in\partial f(\alpha(u))}\left\langle w, \frac{\alpha(u+h)-\alpha(u)}{h}\right\rangle+\epsilon\\
&=&\displaystyle\max_{w\in\partial f(\alpha(u))}\frac{1}{h}\left\langle w, \int_u^{u+h}-\widetilde \vv(\alpha(s))ds\right\rangle+\epsilon\\
&=&\displaystyle-\min_{w\in\partial f(\alpha(u))}\frac{1}{h}\left\langle w, \int_u^{u+h}\widetilde \vv(\alpha(s))ds\right\rangle+\epsilon\\
&\leqslant&\displaystyle-\min_{w\in\partial f(\alpha(u))}\frac{1}{h}\left\langle w, \int_u^{u+h}\widetilde \vv(\alpha(u))ds\right\rangle+\\
&&\displaystyle\max_{w\in\partial f(\alpha(u))}\frac{1}{h}\left\langle w, \int_u^{u+h}(\widetilde \vv(\alpha(s))-\widetilde \vv(\alpha(u)))ds\right\rangle+\epsilon\\
&\leqslant&\displaystyle-\min_{w\in\partial f(\alpha(u))}\langle w, \widetilde \vv(\alpha(u))\rangle+\\
&&\displaystyle\max_{w\in\partial f(\alpha(u))}\frac{\|w\|}{h}\int_u^{u+h}\|\widetilde \vv(\alpha(s))-\widetilde \vv(\alpha(u))\|ds+\epsilon\\
&\leqslant&\displaystyle-\min_{w\in\partial f(\alpha(u))}\langle w, \varphi(\alpha(u)) \vv(\alpha(u))\rangle+\\
&&\displaystyle\max_{w\in\partial f(\alpha(u))}\frac{\|w\|}{h}\int_u^{u+h}K\|\alpha(s)-\alpha(u)\|ds+\epsilon\\
&<&\displaystyle-\varphi(\alpha(u))b+\displaystyle \frac{KK'}{h}\int_u^{u+h}(s-u)ds+\epsilon\\
&=&\displaystyle-\varphi(\alpha(u))b+\displaystyle \frac{KK'}{h}\int_u^{u+h}d\left(\frac{s^2}{2}-us\right)+\epsilon\\
&=&\displaystyle-\varphi(\alpha(u))b+\displaystyle \frac{KK'h}{2}+\epsilon,
\end{array}\end{equation}
where the first inequality follows from the definition of generalized directional derivative, the first equality follows from Lemma~\ref{Clarke}(iii) and the last inequality follows from~\eqref{v} and~\eqref{su}.
Therefore 
\begin{equation}\label{alphauh}f(\alpha(u+h))-f(\alpha(u))<-\varphi(\alpha(u))bh+\frac{KK'h^2}{2}+\epsilon h\leqslant-\frac{\varphi(\alpha(u))}{2}bh<0.\end{equation}
This implies that $f\circ\alpha$ is decreasing at $u.$

It remains to prove the second statement. 
Assume $\alpha(u)\in V.$
In view of~\eqref{varphi}, we have $\varphi(\alpha(u))=1$.
Moreover, $\frac{\alpha(u+h)-\alpha(u)}{h}\in\overline{\mathbb B}^n$ by~\eqref{su}.
Thus, in view of~\eqref{b/4}, for all $h\in \left(0,\min\left\{\frac{b}{2KK'},h_0\right\}\right]$, by replacing the first inequality in~\eqref{long} by the following one
$$\displaystyle\frac{f(\alpha(u+h))-f(\alpha(u))}{h}<\displaystyle f^\circ\left[x;\frac{\alpha(u+h)-\alpha(u)}{h}\right]+\epsilon,$$
where $x\in D$ is a point such that $\dist(\alpha(u),D)=\|\alpha(u)-x\|,$ and repeating the computation in~\eqref{long}, we get 
$$\displaystyle\frac{f(\alpha(u+h))-f(\alpha(u))}{h}<-b+\displaystyle \frac{KK'h}{2}+\epsilon.$$
Consequently
$$f(\alpha(u+h))-f(\alpha(u))<-bh+\frac{KK'h^2}{2}+\epsilon h\leqslant-\frac{bh}{2}.$$
\end{proof}

\begin{proof}[Proof of Proposition~\ref{Prop1}]
Let
\begin{equation*}\label{h}h:=\min\left\{\frac{b}{2KK'},h_0\right\}>0.\end{equation*}
By construction, it is clear that $\displaystyle\sup_{x\in D_k}\dist(x,D)\to 0$ as $k\to +\infty.$ 
Therefore, for $k$ large enough, 
\begin{equation}\label{DkV}D_k\subset V\ \ \text{ and }\ \ \frac{1}{k}<\frac{hb}{2}.\end{equation}
Let us fix such an integer $k$.
For each $t\in[0,1]$, let $\phi(t,s)$ be the (unique) trajectory of $-\widetilde \vv$ with the initial condition $\phi(t,0)=\gamma^k(t),$ i.e.,
$$\phi(t,s)=\gamma^k(t)-\int_0^s\widetilde \vv(\phi(t,s'))d s'.$$
According to~\cite[Theorem 9.5]{Coleman2012}, the mapping $\phi$ is continuous with respect to $t$.
By construction, $x^*,y^*\not\in U$ and $\widetilde \vv$ vanishes outside of $U.$ 
So $$\phi(0,s)=\gamma^k(0)=x^* \text{ and } \phi(1,s)=\gamma^k(0)=y^* \text{ for all } s.$$
Consequently $\phi(\cdot,h)\in\mathcal A$. 
We will show that $f(\phi(t,h))<c$ for any $t\in[0,1]$ which contradicts~\eqref{c}. 
Note that for $t\in[0,1]$ such that $\phi(t,0)=\gamma^k(t)\not\in D_k$, in light of Lemma~\ref{decreasing}, one has 
\begin{equation}\label{notinDk}f(\phi(t,h))\leqslant f(\phi(t,0))=f(\gamma^k(t))<c.\end{equation}
Thus it is enough to consider $t\in[0,1]$ such that $\gamma^k(t)\in D_k.$ Observe that for all such $t$, we have $\varphi(\phi(t,0))=1$ in light of~\eqref{varphi}  and~\eqref{DkV}.
Therefore, by Lemma~\ref{decreasing}, the fact $\gamma^k\in\mathcal A\left(R_0,\frac{1}{k}\right)$ and~\eqref{DkV}, we get 
\begin{equation}\label{inDk}f(\phi(t,h))<f(\phi(t,0))-\frac{bh}{2}<c+\frac{1}{k}-\frac{bh}{2}<c.\end{equation}
Now, from~\eqref{notinDk} and~\eqref{inDk}, it follows that $f(\phi(t,h))<c$ for all $t\in[0,1]$. 
This contradicts the definition of $c$ given by~\eqref{c} and so ends the proof of Proposition~\ref{Prop1}.
\end{proof}

\begin{proposition}\label{Prop2} 
If $\lim_{\epsilon\to 0^+} R(\epsilon)  = +\infty$ then $c\in T_\infty(f)$.
\end{proposition}

Let us make some preparation before proving Proposition~\ref{Prop2}.
Recall that the set $\mathcal A(r,\epsilon)$ and the real number $R(\epsilon)$ are defined respectively by~\eqref{Are} and~\eqref{Re}. 
Let $\epsilon'>0.$ 
For any $\gamma\in \mathcal A(R(\epsilon)+\epsilon',\epsilon),$ we have $\gamma\setminus \mathbb B_{R(\epsilon)}^{n}\ne\emptyset,$ 
so the set 
$$I=I(\gamma):=\{t\in[0,1]:\ \|\gamma(t)\|\geqslant R(\epsilon)\}$$ 
is non-empty.

\begin{lemma}\label{1}
For all $\epsilon>0$ small enough, we have
$$f^{-1}(c+\epsilon)\ne\emptyset\ \text{and}\ R(\epsilon)>\max\{\|x^*\|,\|y^*\|\}.$$ 
Moreover, for all $\epsilon'>0$ (depending on $\epsilon$) small enough,
there is a piecewise linear curve $\gamma\in\mathcal A(R(\epsilon)+\epsilon',\epsilon)$ such that $f(\gamma(t))> c$ for all $t\in I=I(\gamma).$
\end{lemma}

\begin{proof} 
The first statement is clear so let us prove the second one. 
For this, let $\epsilon'>0$ be such that
\begin{equation}\label{xi}2\epsilon'<\min\left\{\dist\left(\left\{f\leqslant c\right\}\cap\overline {\mathbb B}^n_{R(\epsilon)+1},f^{-1}\left(c+{\epsilon}/{2}\right)\right),R(\epsilon)-\|x^*\|,R(\epsilon)-\|y^*\|,2\right\}.\end{equation}
Pick an arbitrary $\beta\in\mathcal A(R(\epsilon)+\epsilon',\epsilon),$ we will deform $\beta$ to get the desired curve. 
Set
$$g(t):=\left\{\begin{array}{lll}
\|\beta(t)\| &\text{if }\ \|\beta(t)\|\leqslant R(\epsilon)-\epsilon'\\
             &\text{or }\ \displaystyle f(\beta(t))\geqslant c+{\epsilon}/{2}\\
\max\left\{
\begin{array}{lll}R(\epsilon)-\epsilon',\\
\displaystyle\|\beta(t)\|-\dist\left(\beta(t),f^{-1}\left(c+{\epsilon}/{2}\right)\right)
\end{array}\right\}
&\text{otherwise}. 
\end{array}\right.$$
Let us show that $g$ is continuous on $[0,1]$. 
Observe that the function 
$$[0,1]\to \mathbb R,\ t\mapsto\max\left\{R(\epsilon)-\epsilon',
\displaystyle\|\beta(t)\|-\dist\left(\beta(t),f^{-1}\left(c+{\epsilon}/{2}\right)\right)\right\}$$
is continuous. Thus, it is clear that $g$ is continuous at any $t\in [0,1]$ such that 
$$\|\beta(t)\|\ne R(\epsilon)-\epsilon'\ \text{ and }\ \displaystyle f(\beta(t))\ne c+{\epsilon}/{2}.$$
It remains to show that $g$ is continuous at any $t\in [0,1]$ such that 
$$\|\beta(t)\|= R(\epsilon)-\epsilon'\ \text{ or }\ \displaystyle f(\beta(t))= c+{\epsilon}/{2}.$$
Firstly, let $t\in [0,1]$ be such that $\|\beta(t)\|= R(\epsilon)-\epsilon'.$ 
Suppose that $t_k\in[0,1]$ is a sequence such that $t_k\to t$ with $\|\beta(t_k)\|>R(\epsilon)-\epsilon',$ we need to show that $g(t_k)\to g(t)$ 
(note that if such a sequence does not exist, then $g$ is continuous at $t$ obviously).
For this, it is enough to assume that 
\begin{equation}\label{>}\|\beta(t_k)\|-\dist\left(\beta(t_k),f^{-1}\left(c+{\epsilon}/{2}\right)\right)> R(\epsilon)-\epsilon'\end{equation}
for all $k$ and show that 
$$\|\beta(t_k)\|-\dist\left(\beta(t_k),f^{-1}\left(c+{\epsilon}/{2}\right)\right)\searrow R(\epsilon)-\epsilon'=\|\beta(t)\|.$$ 
This is equivalent to show that $\dist\left(\beta(t_k),f^{-1}\left(c+\displaystyle{\epsilon}/{2}\right)\right)\to 0,$ i.e., 
$$\dist\left(\beta(t),f^{-1}\left(c+\displaystyle{\epsilon}/{2}\right)\right)= 0.$$
On the other hand, from~\eqref{>}, we get 
$$\|\beta(t_k)\|> R(\epsilon)-\epsilon'+\dist\left(\beta(t_k),f^{-1}\left(c+{\epsilon}/{2}\right)\right).$$
Letting $k\to+\infty,$ we get
$$\|\beta(t)\|\geqslant R(\epsilon)-\epsilon'+\dist\left(\beta(t),f^{-1}\left(c+{\epsilon}/{2}\right)\right)=\|\beta(t)\|+\dist\left(\beta(t),f^{-1}\left(c+{\epsilon}/{2}\right)\right).$$
Hence $\dist\left(\beta(t),f^{-1}\left(c+\displaystyle{\epsilon}/{2}\right)\right)= 0$ and so $g$ is continuous at any $t\in [0,1]$ such that $\|\beta(t)\|= R(\epsilon)-\epsilon'.$
Now we show that $g$ is continuous at any $t\in [0,1]$ such that $\displaystyle f(\beta(t))= c+{\epsilon}/{2}$. 
Let $t_k\in[0,1]$ be a sequence such that $t_k\to t$. 
Without loss of generality, assume that $\|\beta(t)\|> R(\epsilon)-\epsilon'$ and $\|\beta(t_k)\|\ne\|\beta(t)\|$ for all $k$. 
Note that
$$\|\beta(t_k)\|-\dist\left(\beta(t_k),f^{-1}\left(c+{\epsilon}/{2}\right)\right)\to \|\beta(t)\|-\dist\left(\beta(t),f^{-1}\left(c+{\epsilon}/{2}\right)\right)=\|\beta(t)\|.$$
Thus, for $k$ large enough,
$$\|\beta(t_k)\|-\dist\left(\beta(t_k),f^{-1}\left(c+{\epsilon}/{2}\right)\right)>R(\epsilon)-\epsilon'$$
Hence, by definition,
$$g(t_k)=\|\beta(t_k)\|-\dist\left(\beta(t_k),f^{-1}\left(c+{\epsilon}/{2}\right)\right)$$
which yields $g(t_k)\to \|\beta(t)\|.$
Consequently $g$ is continuous at any $t\in [0,1]$ such that $\displaystyle f(\beta(t))= c+{\epsilon}/{2}$ and so is continuous on $[0,1]$.

Set
$$\zeta(t):=\left\{
\begin{array}{lll}g(t)\displaystyle\frac{\beta(t)}{\|\beta(t)\|} &\text{ if } & \beta(t)\ne 0\\
0 &\text{ if } & \beta(t)= 0.
\end{array}\right.$$
We will show that $\zeta$ has the desired properties except being piecewise linear. 
It is clear that $\zeta(t)$ is continuous and 
$$\|\zeta(t)\|\leqslant\|\beta(t)\|\leqslant R(\epsilon)+\epsilon'\ \text{ for any }\ t\in[0,1].$$ 
By~\eqref{xi} and by the definition of the function $g$, it follows that 
$\zeta(0)=\beta(0)=x^*$ and $\zeta(1)=\beta(1)=y^*$. 
Moreover, for all $t\in[0,1]$ such that $\|\beta(t)\|>R(\epsilon)-\epsilon'$ and $\displaystyle f(\beta(t))<c+{\epsilon}/{2},$ we have
$$\begin{array}{lll}
\|\beta(t)-\zeta(t)\|=\|\beta(t)\|-g(t)&=&\min\left\{\|\beta(t)\|-(R(\epsilon)-\epsilon'),\dist\left(\beta(t),f^{-1}\left(c+\displaystyle{\epsilon}/{2}\right)\right)\right\}\\
&\leqslant& \dist\left(\beta(t),f^{-1}\left(c+\displaystyle{\epsilon}/{2}\right)\right).
\end{array}$$
This, together with the fact  $f(\beta(t))< c+\displaystyle{\epsilon}/{2}$, implies $f(\zeta(t))\leqslant c+\displaystyle{\epsilon}/{2}.$ 
Consequently $\zeta\in\mathcal A(R(\epsilon)+\epsilon',\epsilon).$ 
Now we show that $f(\zeta(t))> c$ for all $t\in I(\zeta).$ 
Assume that $\|\zeta(t)\|\geqslant R(\epsilon)$, then $\|\beta(t)\|\geqslant R(\epsilon)$ and so
$$\epsilon'>\|\beta(t)-\zeta(t)\|=\dist\left(\beta(t),f^{-1}\left(c+{\epsilon}/{2}\right)\right).$$
Combining this with~\eqref{xi} gives
$$\begin{array}{lll}\dist\left(\zeta(t),f^{-1}\left(c+\displaystyle{\epsilon}/{2}\right)\right)&\leqslant& \|\beta(t)-\zeta(t)\|+\dist\left(\beta(t),f^{-1}\left(c+\displaystyle{\epsilon}/{2}\right)\right)\\
&<&2\epsilon'<\dist\left(\left\{f\leqslant c\right\}\cap\overline {\mathbb B}^n_{R(\epsilon)+1},f^{-1}\left(c+\displaystyle{\epsilon}/{2}\right)\right)\\
&\leqslant&\dist\left(\left\{f\leqslant c\right\}\cap\overline {\mathbb B}^n_{R(\epsilon)+\epsilon'},f^{-1}\left(c+\displaystyle{\epsilon}/{2}\right)\right)
\end{array}$$
which yields $f(\zeta(t))>c.$ 

Finally, we need to deform $\zeta$ to get the desired curve. 
Set 
$$\lambda:=\min\left\{c+\epsilon-\max_{t\in[0,1]}f(\zeta(t)),\min_{t\in I(\zeta)}f(\zeta(t))-c\right\}>0.$$
Since $\zeta$ and $f$ are continuous on $[0,1]$ and $\overline{\mathbb B}^n_{R(\epsilon)+\epsilon'}$, respectively, they are uniformly continuous on the respective sets by the Heine--Cantor theorem.
Thus there are constants $\nu,\nu'>0$ such that 
$$\|\zeta(t)-\zeta(s)\|<\nu\ \text{ for } t,s\in[a,b] \text{ with } |t-s|<\nu'$$ 
and 
$$|f(x)-f(y)|<\lambda\ \text{ for } x,y\in \overline{\mathbb B}^n_{R(\epsilon)+\epsilon'} \text{ with } \|x-y\|<\nu.$$
Let $a=:T_0<T_1<\dots<T_q:=b$ be a finite sequence such that $|T_i-T_{i-1}|<\nu',$ for $i=1,\dots,q$
and let $\gamma\colon[0,1]\to \overline{\mathbb B}^n_{R(\epsilon)+\epsilon'}$ be the continuous piecewise linear curve defined by the sequence $\{\zeta(T_0),\dots,\zeta(T_q)\}$ so $\gamma(0)=\zeta(0)$, $\gamma(1)=\zeta(1)$. It is not hard to check the following facts:
\begin{enumerate}[{\rm (a)}]
\item $\|\gamma(t)\|\leqslant R(\epsilon)+\epsilon' \text{ for any } t\in[0,1]$;
\item $\max\{|f(\gamma(t))-f(\zeta(T_{i-1}))|,|f(\gamma(t))-f(\zeta(T_{i}))|\}<\lambda$ for $t\in[T_{i-1},T_{i}]\ (i=1,\dots,q).$
\end{enumerate}
For  all $t\in [0,1],$ let $i$ be such that $t\in [T_{i-1},T_i].$
By (b), we have
$$\begin{array}{lll}f(\gamma(t))&\leqslant&f(\gamma(T_i))+|f(\gamma(t))-f(\gamma(T_i))|=f(\zeta(T_i))+|f(\gamma(t))-f(\zeta(T_i))|\\
&<&f(\zeta(T_i))+\lambda\leqslant f(\zeta(T_i))+c+\epsilon-\max_{s\in[0,1]}f(\zeta(s))\leqslant c+\epsilon.
\end{array}$$
Combining this with (a) yields 
$$\gamma\in\mathcal A(R(\epsilon)+\epsilon',\epsilon).$$ 
We will show that $f(\gamma(t))> c$ for all $t\in I=I(\gamma),$ which ends the proof of the lemma.
Pick arbitrarily $t\in I$ and assume that $[\zeta(T_{i-1}),\zeta(T_i)]$ is the line segment containing $\gamma(t).$ 
Since $\|\gamma(t)\|\geqslant R(\epsilon),$ it follows that $\max\{\|\zeta(T_{i-1})\|,\|\zeta(T_{i})\|\}\geqslant R(\epsilon).$
Without loss of generality, assume that $\|\zeta(T_{i})\|\geqslant R(\epsilon),$ so $\zeta(T_{i})\in I(\zeta).$
From this and (b), we get 
$$\begin{array}{lll}f(\gamma(t))&\geqslant&f(\gamma(T_i))-|f(\gamma(t))-f(\gamma(T_i))|=f(\zeta(T_i))-|f(\gamma(t))-f(\zeta(T_i))|\\
&>&f(\zeta(T_i))-\lambda\geqslant f(\zeta(T_i))-\min_{t\in I(\zeta)}f(\zeta(t))-c\geqslant c.
\end{array}$$
The lemma is proved.
\end{proof}

Now the proof needs the following key lemma.

\begin{lemma}\label{key} Let $\epsilon>0$ be as in Lemma~\ref{1}. Then there are $x\in X:=f^{-1}([c,c+\epsilon])\cap\mathbb S^{n-1}_{R(\epsilon)}$ and $v\in\partial f(x)$ such that $x$ and $v$ are linear dependent.
\end{lemma}
\begin{proof} 
The construction in the proof is illustrated in the figure below.

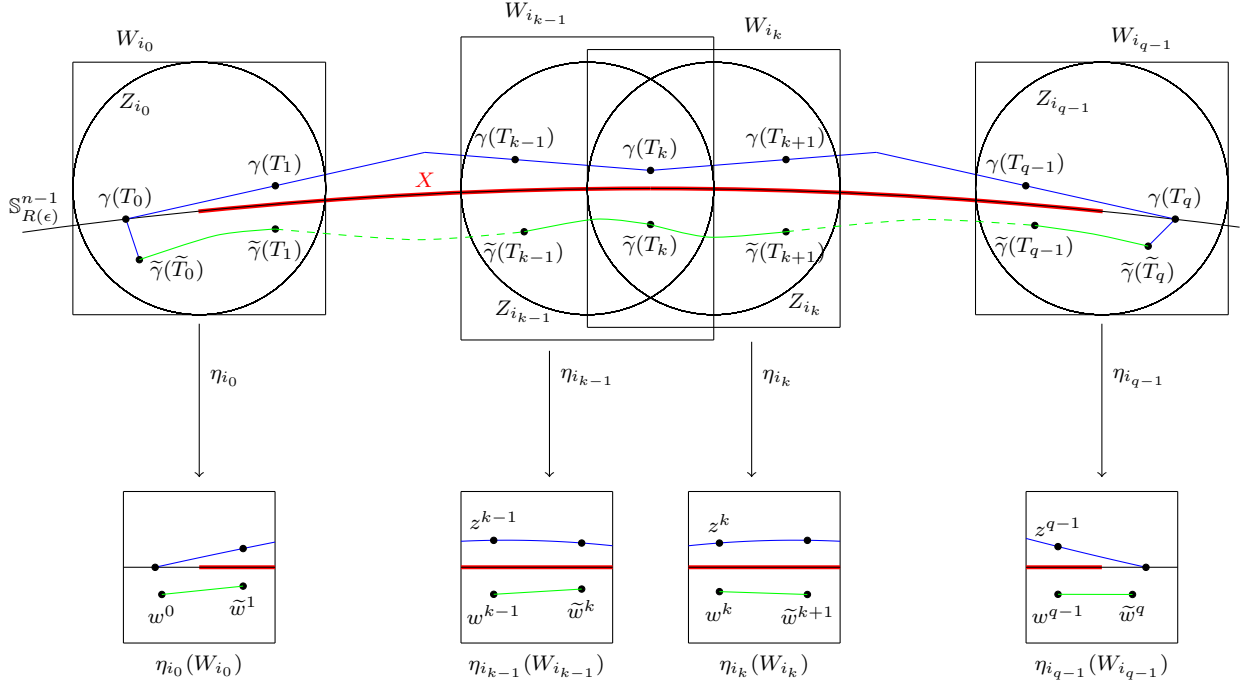
\begin{center}
\begin{tikzpicture}[scale=1.2]
	\def\RR{50};
	\draw[line width=1.5pt,red](0,0) arc (90:84.25:\RR);
	\draw[line width=1.5pt,red](0,0) arc (90:95.75:\RR);
	\draw(0,0) arc (90:98:\RR);
	\draw(0,0) arc (90:82.5:\RR);

	\def\a{0.025};\def\b{0.045};\def\c{0.34};\def\i{2.5};\def\r{1.4};\def\f{1.1};\def\g{1.2};\def\k{0.58};\def\h{0.4};\def\j{0.15};
	\draw[domain=0:1,variable=\t,blue] plot ({-\k*\r-2*\i+\t*(\k*\r+\i)},{-\c+\t*(\h+\c)});
	\draw[domain=0:1,variable=\t,blue] plot ({-\i+\t*\i},{\h-\t*\h/2});
	\draw[domain=0:1,variable=\t,blue] plot ({\i-\t*\i},{\h-\t*\h/2});
	\draw[domain=0:1,variable=\t,blue] plot ({\k*\r+2*\i-\t*(\k*\r+\i)},{-\c+\t*(\h+\c)});
	
	\draw[domain=0:1,variable=\t,blue] plot ({-\k*\r-2*\i+\t*\j},{-\c-3*\t*\j});
	\draw[domain=0:1,variable=\t,blue] plot ({\k*\r+2*\i-\t*2*\j},{-\c-2*\t*\j});

	\draw(-2*\i+\r,0) arc (0:3600:\r);
	\draw( 2*\i+\r,0) arc (0:3600:\r);
	\draw( 1.5*\r,0) arc (0:3600:\r);
	\draw( 0.5*\r,0) arc (0:3600:\r);
	
	\draw(-2*\i+\r,-\r)--(-2*\i+\r,\r);	
		\draw(-2*\i-\r,-\r)--(-2*\i-\r,\r);	
		\draw(-2*\i+\r,-\r)--(-2*\i-\r,-\r);	
		\draw(-2*\i+\r, \r)--(-2*\i-\r,\r);

	\draw(2*\i+\r,-\r)--(2*\i+\r,\r);	
		\draw(2*\i-\r,-\r)--(2*\i-\r,\r);	
		\draw(2*\i+\r,-\r)--(2*\i-\r,-\r);	
		\draw(2*\i+\r, \r)--(2*\i-\r,\r);
		
	\draw(1.5*\r,-\r*\f)--(1.5*\r,\r*\f);	
		\draw(-0.5*\r,-\r*\f)--(-0.5*\r,\r*\f);	
		\draw(1.5*\r, -\r*\f)--(-0.5*\r,-\r*\f);	
		\draw(1.5*\r,  \r*\f)--(-0.5*\r,\r*\f);
		
	\draw(-1.5*\r,-\r*\g)--(-1.5*\r,\r*\g);	
		\draw(0.5*\r,-\r*\g)--(0.5*\r,\r*\g);	
		\draw(-1.5*\r, -\r*\g)--(0.5*\r,-\r*\g);	
		\draw(-1.5*\r,  \r*\g)--(0.5*\r,\r*\g);
		
	\def\t{-\i};
	\coordinate (X) at ({\t},{-0.1});		\node [above] at (X) {\color{red}\tiny $X$};
										    \node [below] at (-2*\i-1.3*\r,0.05*\r) {\tiny $\mathbb S^{n-1}_{R(\epsilon)}$};
						    
	\def\t{1};
	\coordinate (Gt0) at ({-\k*\r-2*\i+\t*\j},{-\c-3*\t*\j});	\filldraw (Gt0) circle (1pt);	\node [right] at ($(Gt0)-(0,0.1)$) {\tiny $\widetilde\gamma(\widetilde T_0)$};
	\def\t{0};
	\coordinate (G0) at ({-\k*\r-2*\i+\t*(\k*\r+\i)},{-\c+\t*(\h+\c)});	\filldraw (G0) circle (1pt);	\node [above] at (G0) {\tiny $\gamma(T_0)$};
	\def\t{0.5};
	\coordinate (G1) at ({-\k*\r-2*\i+\t*(\k*\r+\i)},{-\c+\t*(\h+\c)});	\filldraw (G1) circle (1pt);	\node [above] at (G1) {\tiny$\gamma(T_1)$};
	\coordinate (Gt1) at ({-\k*\r-2*\i+\t*(\k*\r+\i)},{-\c+\t*(\h+\c)-1.2*\h});	\filldraw (Gt1) circle (1pt);	\node [below] at (Gt1) {\tiny$\widetilde\gamma(T_1)$};
	\draw [green] plot [smooth] coordinates {(Gt0) ($0.5*(Gt0)+0.5*(Gt1)+(0.1,0.1)$) (Gt1)};
	
	\def\t{0.4};	
	\coordinate (Gk-1) at ({-\i+\t*\i},{\h-\t*\h/2});					\filldraw (Gk-1) circle (1pt);	\node [above] at (Gk-1) {\tiny$\gamma(T_{k-1})$};
	\coordinate (Gtk-1) at ({-\i+\t*\i+0.1},{-\h-\t*\h/2});					\filldraw (Gtk-1) circle (1pt);	\node [below] at (Gtk-1) {\tiny$\widetilde\gamma(T_{k-1})$};
	\draw [green,dashed] plot [smooth] coordinates {(Gt1) ($0.4*(Gt1)+0.6*(Gtk-1)-(0.1,0.1)$) (Gtk-1)};
	\def\t{1};
	\coordinate (Gk) at ({-\i+\t*\i},{\h-\t*\h/2});						\filldraw (Gk) circle (1pt);	\node [above] at (Gk) {\tiny$\gamma(T_{k})$};
	\coordinate (Gtk) at ({-\i+\t*\i},{-\h/2-\t*\h/2});					\filldraw (Gtk) circle (1pt);	\node [below] at (Gtk) {\tiny$\widetilde\gamma(T_{k})$};
	\draw [green] plot [smooth] coordinates {(Gtk-1) ($0.5*(Gtk-1)+0.5*(Gtk)+(0.1,0.1)$) (Gtk)};

	\def\t{0.4};
	\coordinate (Gk+1) at ({\i-\t*\i},{\h-\t*\h/2});					\filldraw (Gk+1) circle (1pt);	\node [above] at (Gk+1) {\tiny$\gamma(T_{k+1})$};
	\coordinate (Gtk+1) at ({\i-\t*\i},{-\h-\t*\h/2});					\filldraw (Gtk+1) circle (1pt);	\node [below] at (Gtk+1) {\tiny$\widetilde\gamma(T_{k+1})$};
	\draw [green] plot [smooth] coordinates {(Gtk) ($0.5*(Gtk)+0.5*(Gtk+1)-(0.1,0.1)$) (Gtk+1)};
	
	\def\t{0.5};
	\coordinate (Gq-1) at ({\k*\r+2*\i-\t*(\k*\r+\i)},{-\c+\t*(\h+\c)});\filldraw (Gq-1) circle (1pt);	\node [above] at (Gq-1) {\tiny$\gamma(T_{q-1})$};
	\coordinate (Gtq-1) at ({\k*\r+2*\i-\t*(\k*\r+\i)+0.1},{-\c+\t*(-1.2*\h+\c)});\filldraw (Gtq-1) circle (1pt);\node [below] at (Gtq-1) {\tiny$\widetilde\gamma(T_{q-1})$};
	\draw [green,dashed] plot [smooth] coordinates {(Gtk+1) ($0.5*(Gtk+1)+0.5*(Gtq-1)+(0.1,0.1)$) (Gtq-1)};
	
	\def\t{0};
	\coordinate (Gq) at ({\k*\r+2*\i-\t*(\k*\r+\i)},{-\c+\t*(\h+\c)});	\filldraw (Gq) circle (1pt);	\node [above] at (Gq) {\tiny$\gamma(T_{q})$};
	\def\t{1};
	\coordinate (Gtq) at ({\k*\r+2*\i-\t*2*\j},{-\c-2*\t*\j});			\filldraw (Gtq) circle (1pt);	\node [below] at (Gtq) {\tiny$\widetilde\gamma(\widetilde T_{q})$};
	\draw [green] plot [smooth] coordinates {(Gtq-1) ($0.5*(Gtq-1)+0.5*(Gtq)+(0.02,0.02)$) (Gtq)};
	
	\def\t{-2*\i};
	\coordinate (Wi0) at ({\t-\r/2},{\r});		\node [above] at (Wi0) {\tiny$W_{i_0}$};
	\def\t{-0.9*\r};	
	\coordinate (Wik-1) at ({\t},{\g*\r});		\node [above] at (Wik-1) {\tiny$W_{i_{k-1}}$};
	\def\t{0.9*\r};
	\coordinate (Wik) at ({\t},{\f*\r});		\node [above] at (Wik) {\tiny$W_{i_{k}}$};
	\def\t{1.9*\i};
	\coordinate (Wiq-1) at ({\t+\r/2},{\r});	\node [above] at (Wiq-1) {\tiny$W_{i_{q-1}}$};
	
	\def\t{2*\i};
	\coordinate (Zi0) at ({-\t-\r/2},{\r/2});	\node [above] at (Zi0) {\tiny$Z_{i_0}$};
	\def\t{-\r};
	\coordinate (Zik-1) at ({\t},{-0.8*\r});	\node [below] at (Zik-1) {\tiny$Z_{i_{k-1}}$};
	\def\t{\r};
	\coordinate (Zik) at ({\t},{-0.75*\r});		\node [below right] at (Zik) {\tiny$Z_{i_{k}}$};
	\def\t{2*\i};
	\coordinate (Ziq-1) at ({\t-0.3*\r},{\r/2});	\node [above] at (Ziq-1) {\tiny$Z_{i_{q-1}}$};
		
	\def\t{2*\i};
	\draw [->] (-\t,-\r-0.1)--(-\t,-2*\r-0.4);  	 \node [right] at (-\t,-1.5*\r) {\tiny$\eta_{i_0}$};
	\draw [->] ( \t,-\r-0.1)--( \t,-2*\r-0.4);  	 \node [right] at ( \t,-1.5*\r) {\tiny$\eta_{i_{q-1}}$};
	\def\t{0.8*\r};
	\draw [out=-90,in=90,->] (-\t,-\r-0.2+1-\g) to (-\t,-2*\r-0.4);  \node [right] at (-\t,-1.5*\r) {\tiny$\eta_{i_{k-1}}$};
	\draw [->] (\t,-\r-0.2+1-\f) --(\t,-2*\r-0.4);   \node [right] at (\t,-1.5*\r) {\tiny$\eta_{i_{k}}$};
		
	\def\p{3*\r};\def\e{0.6};\def\h{0.3};
	\def\t{2*\i};\def\RR{5};
	\draw(-\t+\r*\e,-\r*\e-\p)--(-\t+\r*\e,\r*\e-\p);	
		\draw(-\t-\r*\e,-\r*\e-\p)--(-\t-\r*\e,\r*\e-\p);	
		\draw(-\t+\r*\e,-\r*\e-\p)--(-\t-\r*\e,-\r*\e-\p);	
		\draw(-\t+\r*\e, \r*\e-\p)--(-\t-\r*\e,\r*\e-\p);
		\coordinate (EWi0) at ({-\t},{-\r*\e-\p});		\node [below] at (EWi0) {\tiny$\eta_{i_0}(W_{i_0})$};		
			\draw[line width=1.5pt,red](-\t,-\p)--(-\t+\r*\e,-\p);	
			\draw(-\t-\r*\e,-\p)--(-\t+\r*\e,-\p);	
			\draw[domain=-\t-\k*\r*\e:-\t+\r*\e, variable=\s,blue] plot ({\s},{sqrt(\s+\t+\k*\r*\e+\RR)-\p-sqrt(\RR)});
			\def\s{-\t-\k*\r*\e};
			\coordinate (z0) at ({\s},{sqrt(\s+\t+\k*\r*\e+\RR)-\p-sqrt(\RR)});	\filldraw (z0) circle (1pt);
			\def\s{-\t+\k*\r*\e};
			\coordinate (z1) at ({\s},{sqrt(\s+\t+\k*\r*\e+\RR)-\p-sqrt(\RR)});	\filldraw (z1) circle (1pt);
			\coordinate (w0) at ($(z0)-(-\h/4,\h)$);							\filldraw (w0) circle (1pt); \node [below] at (w0) {\tiny$w^0$};
			\coordinate (wt1) at ({\s},{-sqrt(\s+\t+\k*\r*\e+\RR)-\p+sqrt(\RR)});	\filldraw (wt1) circle (1pt); \node [below] at (wt1) {\tiny$\widetilde w^1$};
			\draw [green] (w0)--(wt1);
			
	\draw(\t+\r*\e,-\r*\e-\p)--(\t+\r*\e,\r*\e-\p);	
		\draw(\t-\r*\e,-\r*\e-\p)--(\t-\r*\e,\r*\e-\p);	
		\draw(\t+\r*\e,-\r*\e-\p)--(\t-\r*\e,-\r*\e-\p);	
		\draw(\t+\r*\e, \r*\e-\p)--(\t-\r*\e,\r*\e-\p);
		\coordinate (EWiq-1) at ({\t},{-\r*\e-\p});		\node [below] at (EWiq-1) {\tiny$\eta_{i_{q-1}}(W_{i_{q-1}})$};		
			\draw[line width=1.5pt,red](\t-\r*\e,-\p)--(\t,-\p);	
			\draw(\t-\r*\e,-\p)--(\t+\r*\e,-\p);	
			\draw[domain=\t-\r*\e:\t+\k*\r*\e, variable=\s,blue] plot ({\s},{-sqrt(\s-\t-\k*\r*\e+\RR)-\p+sqrt(\RR)});
			\def\s{\t-\k*\r*\e};
			\coordinate (zq-1) at ({\s},{-sqrt(\s-\t-\k*\r*\e+\RR)-\p+sqrt(\RR)});	\filldraw (zq-1) circle (1pt);	\node [above] at (zq-1) {\tiny$z^{q-1}$};
			\coordinate (wq-1) at ({\s},{ \b*(\s-\t+\c)*(\s-\t+\c)-\p-\h});			\filldraw (wq-1) circle (1pt);	\node [below] at (wq-1) {\tiny$w^{q-1}$};
			\def\s{\t+\k*\r*\e};
			\coordinate (zq) at ({\s},{-sqrt(\s-\t-\k*\r*\e+\RR)-\p+sqrt(\RR)});	\filldraw (zq) circle (1pt);	
			\coordinate (wtq) at ($(zq)-(\h/2,\h)$);								\filldraw (wtq) circle (1pt);	\node [below] at (wtq) {\tiny$\widetilde w^{q}$};
			\draw [green] (wq-1)--(wtq);
		
	\draw(-2.5*\r*\e,-\r*\e-\p)--(-2.5*\r*\e,\r*\e-\p);	
		\draw(-0.5*\r*\e,-\r*\e-\p)--(-0.5*\r*\e,\r*\e-\p);	
		\draw(-2.5*\r*\e, -\r*\e-\p)--(-0.5*\r*\e,-\r*\e-\p);	
		\draw(-2.5*\r*\e,  \r*\e-\p)--(-0.5*\r*\e,\r*\e-\p);
		\coordinate (EWik-1) at ({-1.5*\r*\e},{-\r*\e-\p});		\node [below] at (EWik-1) {\tiny$\eta_{i_{k-1}}(W_{i_{k-1}})$};				
			\def\t{0.5*\i};
			\draw[line width=1.5pt,red](-0.5*\r*\e,-\p)--(-2.5*\r*\e,-\p);	
			\draw(-0.5*\r*\e,-\p)--(-2.5*\r*\e,-\p);	
			\draw[domain=-2.5*\r*\e:-0.5*\r*\e, variable=\s,blue] plot ({\s},{-\b*(\s+\t+\c)*(\s+\t+\c)-\p+\h});
			\def\s{-\t-\k*\r*\e};
			\coordinate (zk-1) at ({\s},{-\b*(\s+\t+\c)*(\s+\t+\c)-\p+\h});	\filldraw (zk-1) circle (1pt);	\node [above] at (zk-1) {\tiny$z^{k-1}$};
			\coordinate (wk-1) at ({\s},{ \b*(\s+\t+\c)*(\s+\t+\c)-\p-\h});	\filldraw (wk-1) circle (1pt);	\node [below] at (wk-1) {\tiny$w^{k-1}$};
			\def\s{-\t+\k*\r*\e};
			\coordinate (zk) at ({\s},{-\b*(\s+\t+\c)*(\s+\t+\c)-\p+\h});	\filldraw (zk) circle (1pt);	
			\coordinate (wtk) at ({\s},{-\b*(\s+\t+\c)*(\s+\t+\c)-\p-0.7*\h});	\filldraw (wtk) circle (1pt);	\node [below] at (wtk) {\tiny$\widetilde w^{k}$};
			\draw [green] (wk-1)--(wtk);
				
	\draw(2.5*\r*\e,-\r*\e-\p)--(2.5*\r*\e,\r*\e-\p);	
		\draw(0.5*\r*\e,-\r*\e-\p)--(0.5*\r*\e,\r*\e-\p);	
		\draw(2.5*\r*\e, -\r*\e-\p)--(0.5*\r*\e,-\r*\e-\p);	
		\draw(2.5*\r*\e,  \r*\e-\p)--(0.5*\r*\e,\r*\e-\p);
		\coordinate (EWik) at ({1.5*\r*\e},{-\r*\e-\p});		\node [below] at (EWik) {\tiny$\eta_{i_k}(W_{i_k})$};		
			\def\t{0.5*\i};
			\draw[line width=1.5pt,red](0.5*\r*\e,-\p)--(2.5*\r*\e,-\p);	
			\draw(0.5*\r*\e,-\p)--(2.5*\r*\e,-\p);	
			\draw[domain=0.5*\r*\e:2.5*\r*\e, variable=\s,blue] plot ({\s},{-\b*(\s-\t-\c)*(\s-\t-\c)-\p+\h});
			\def\s{\t-\k*\r*\e};
			\coordinate (zk) at ({\s},{-\b*(\s-\t-\c)*(\s-\t-\c)-\p+\h});	\filldraw (zk) circle (1pt);	\node [above] at (zk) {\tiny$z^{k}$};
			\coordinate (wk) at ({\s},{ \b*(\s-\t-\c)*(\s-\t-\c)-\p-\h});	\filldraw (wk) circle (1pt);	\node [below] at (wk) {\tiny$w^{k}$};
			\def\s{\t+\k*\r*\e};
			\coordinate (zk+1) at ({\s},{-\b*(\s-\t-\c)*(\s-\t-\c)-\p+\h});	\filldraw (zk+1) circle (1pt);	
			\coordinate (wtk+1) at ({\s},{-\b*(\s-\t-\c)*(\s-\t-\c)-\p-\h});	\filldraw (wtk+1) circle (1pt);	\node [below] at (wtk+1) {\tiny$\widetilde w^{k+1}$};
			\draw [green] (wk)--(wtk+1);
\end{tikzpicture}
\end{center}

Consider the mapping $F\colon\mathbb R^n\to\mathbb R^2$ defined by $F(x):=(f(x),\|x\|^2).$ In light of~\cite[Proposition 2.6.2(e)]{Clarke1990}, for $x\ne 0,$ we have 
$$\partial F(x)\subset
\partial f(x)\times \{2x\}=\left\{\left(
\begin{array}{cl}
v\\
2x
\end{array}
\right):\ v\in\partial f(x)\right\}.$$
Assume for contradiction that $x$ and $v$ are linearly independent for all $x\in X$ and all $v\in\partial f(x).$
Then we have $\rank(w)=2$ for all $x\in X$ and all $w\in\partial F(x).$
In view of Lemma~\ref{CR}, for each $x\in X$, there exist an open neighborhood $W_x$ of $x$ and a bi-Lipschitz homeomorphism $\eta_x\colon W_x\to\eta_x(W_x)\subset\mathbb R^n$ with $\eta_x(x)=0$ such that 
\begin{equation}\label{Feta}F\circ\eta_x^{-1}(u_1,u_2,\dots,u_n)=(u_1,u_2)+F(x)=(u_1,u_2)+(f(x),R(\epsilon)^2),\end{equation}
for all $u=(u_1,\dots,u_n)\in\eta_x(W_x).$ 
For each $x$, shrinking $W_x$ if necessary so that $x^*,y^*\not\in W_x$.
It is clear that there is $\rho_x>0$ such that 
$$B_x:=\{u\in\mathbb R^n:\ |u_i|<\rho_x,\ i=1,\dots,n\}\subset\eta_x(W_x).$$
Shrinking $W_x$ more if necessary so that $B_x=\eta_x(W_x)$. 
By construction, for any $(u_1,\dots,u_n)\in B_x,$~\eqref{Feta} is equivalent to the following equalities
\begin{equation}\label{Feta1}f(\eta_x^{-1}(u_1,u_2,\dots,u_n))=u_1+f(x) \ \text{ and }\ \|\eta_x^{-1}(u_1,u_2,\dots,u_n)\|^2=u_2+R(\epsilon)^2.\end{equation}
Let $Z_x\subset W_x$ be an open ball centered at $x.$
As $\{Z_x:\ x\in X\}$ is an open cover of $X$, by compactness, there are distinct points $\{x^1,\dots,x^p\}\in X$ such that 
$$X\subset\bigcup_{i=1}^p Z_{x^i}.$$ 
Let $\lambda>0$ be the constant depending on the cover $\{Z_{x^i}:\ i=1,\dots,p\}$ given by Lemma~\ref{Cover} and $L\geqslant 1$ be a common Lipschitz constant for $\eta_1:=\eta_{x^{1}},\dots,\eta_p:=\eta_{x^{p}},\eta_1^{-1},\dots,\eta_p^{-1}.$ 
Set $\nu:=\frac{\lambda}{2L^2}.$ 
Let $\epsilon'>0$ and $\gamma$ be, respectively, the constant and the piecewise linear curve depending on $\epsilon'$ determined by Lemma~\ref{1}. 
We will show that if $\epsilon'>0$ is small enough, then 
\begin{equation}\label{nu}\gamma\setminus \mathbb B_{R(\epsilon)}^{n}\subset N_\nu(f^{-1}[c,c+\epsilon]\cap\mathbb S^{n-1}_{R(\epsilon)}).\end{equation}
Indeed, assume for contradiction that for all $k>0$ large enough, there is $\gamma^k\in\mathcal A\left(R(\epsilon)+\frac{1}{k},\epsilon\right)$ and
$$y^k\in\gamma^k\setminus (\mathbb B_{R(\epsilon)}^{n}\cup N_\nu(f^{-1}[c,c+\epsilon]\cap\mathbb S^{n-1}_{R(\epsilon)}))\ne\emptyset.$$ 
Observe that $R(\epsilon)\leqslant \|y^k\|\leqslant R(\epsilon)+\frac{1}{k}$.  
So by compactness, the sequence $y^k$ has at least a cluster point, say $y^0.$ Clearly $\|y^0\|=R(\epsilon)$. 
In addition, as $c<f(y^k)\leqslant c+\epsilon$ in view of Lemma~\ref{1}, one has $c\leqslant f(y^0)\leqslant c+\epsilon.$
Hence $$y^0\in f^{-1}[c,c+\epsilon]\cap\mathbb S^{n-1}_{R(\epsilon)}.$$
On the other hand, since $y^k\not\in N_\nu(f^{-1}[c,c+\epsilon]\cap\mathbb S^{n-1}_{R(\epsilon)})$, it follows that 
$$\dist(y^k, f^{-1}[c,c+\epsilon]\cap\mathbb S^{n-1}_{R(\epsilon)})>\nu.$$ 
By letting $k\to+\infty$, we get
$$\dist(y^0, f^{-1}[c,c+\epsilon]\cap\mathbb S^{n-1}_{R(\epsilon)})\geqslant\nu,$$
which is a contradiction. 
Therefore~\eqref{nu} must hold for $\epsilon'>0$ sufficiently small. 
We will deform $\gamma$ to get an other piecewise linear curve $\widetilde\gamma\in \mathcal A(R(\epsilon),\epsilon)$ such that $\|\widetilde\gamma(t)\|<R(\epsilon)$ for any $t\in[0,1]$. 
As $\gamma$ is piecewise linear, $\gamma\setminus \mathbb B_{R(\epsilon)}^{n}$ has finitely many connected components. 
Without loss of generality, assume that $\gamma\setminus \mathbb B_{R(\epsilon)}^{n}$ is connected.
In fact, if $\gamma\setminus \mathbb B_{R(\epsilon)}^{n}$ is not connected, then it is enough to apply the process below on each connected component of $\gamma\setminus \mathbb B_{R(\epsilon)}^{n}$. 

Let $[a,b]\subset(0,1)$ be the interval such that $\gamma(t)\geqslant R(\epsilon)$ if and only if $t\in [a,b].$ 
In view of Lemma~\ref{Decompose}, there exist finite sequences $a=:T_0<\cdots<T_q:=b$ and $i_0\ne i_1\ne\cdots\ne i_{q-1}$ with $i_k\in\{1,\dots,p\}\ (k=0,\dots,q)$ such that:
\begin{enumerate}[{\rm (a)}]
\item $\gamma[T_{k-1},T_{k}]\subset W_{i_{k-1}}$ for $k=1,\dots,q$; and
\item $\eta_{i_k}^{-1}(w^k)\in W_{i_{k-1}}\cap W_{i_{k}}$ for $k=1,\dots,q-1,$ where 
$$z^k:=\eta_{i_k}(\gamma(T_k))\ \text{ and }\ w^k:=(z^k_1,w^k_2,z^k_3,\dots,z^k_n)$$
with
$$w_2^k=\left\{\begin{array}{llll}
-z_2^k             &\text{if}& z_2^k\ne 0\\
-\frac{\lambda}{L} &\text{if}& z_2^k=   0.
\end{array}\right.$$
\end{enumerate}
Since $\gamma$ is continuous, for $\delta>0$ small enough, we have  
$$a-\delta\in[0,1],\ b+\delta\in[0,1],$$ 
$$\gamma(a-\delta)\in W_{i_0}\ \text{and}\ \gamma(b+\delta)\in W_{i_{q-1}}.$$
Let 
$$a-\delta=:\widetilde T_0<\widetilde T_1:=T_1<\dots<\widetilde T_{q-1}:=T_{q-1}<\widetilde T_q:=b+\delta,$$
\begin{equation}\label{0q}
w^0:=\eta_{i_0}(\gamma(\widetilde T_0))\in \eta_{i_0}(W_{i_0})=B_{x_{i_0}},\ \widetilde w^q:=\eta_{i_{q-1}}(\gamma(\widetilde T_{q}))\in\eta_{i_{q-1}}(W_{i_{q-1}})=B_{x_{i_{q-1}}}
\end{equation}
and 
$$\widetilde w^k:=\eta_{i_{k-1}}(\eta_{i_{k}}^{-1}(w^{k}))\ \text{ for }\ k=1,\dots,q-1.$$
Define the curve
$$\widetilde\gamma(t):=\left\{
\begin{array}{llll}
\gamma(t) & \text{if} & t\in [0,1]\setminus (\widetilde T_{0},\widetilde T_{q})\\
\eta_{i_{k-1}}^{-1}\left[w^{k-1}+\frac{t-\widetilde T_{k-1}}{\widetilde T_{{k}}-\widetilde T_{{k-1}}}(\widetilde w^{k}- w^{{k-1}})\right] & \text{if} & t\in [\widetilde T_{{k-1}},\widetilde T_{{k}}]\ (k=1,\dots,q).
\end{array}\right.
$$
We note the following facts:
\begin{itemize}
\item $\widetilde\gamma(\widetilde T_0)=\eta_{i_0}^{-1}(w^0)=\gamma(\widetilde T_{0}),$
\item $\eta_{i_k}^{-1}(w^k)=\eta_{i_{k-1}}^{-1}(\widetilde w^{k})$ for $k=1,\dots,q-1,$
\item $\widetilde\gamma(\widetilde T_{q})=\eta_{i_{q-1}}^{-1}(\widetilde w^q)=\gamma(\widetilde T_{q}),$ and
\item $w^{k-1},\widetilde w^{k}\in B_{x_{i_k}}$ for $k=1,\dots,q.$
\end{itemize}
Thus $\widetilde\gamma$ is continuous. 
We will prove that $\|\widetilde\gamma(t)\|<R(\epsilon)$ and $f(\widetilde\gamma(t))<c+\epsilon$ for any $t\in[0,1]$, which contradicts the definition of $R(\epsilon).$
Clearly, 
$$\|\widetilde\gamma(t)\|<R(\epsilon)\ \ \text { and }\ \ f(\widetilde\gamma(t))<c+\epsilon\ \text{ for all }\ t\in[0,1]\setminus (\widetilde T_{0},\widetilde T_{q}).$$
So it remains to show that $\|\widetilde\gamma(t)\|<R(\epsilon)$ and $f(\widetilde\gamma(t))<c+\epsilon$ for all $t\in(\widetilde T_{0},\widetilde T_{q}).$

This and~\eqref{0q} implies
$$\|\eta_{i_0}^{-1}(w^0)\|=\|\gamma(\widetilde T_{0})\|<R(\epsilon)\ \text{ and }\ \|\eta_{i_{q-1}}^{-1}(\widetilde w^q)\|=\|\gamma(\widetilde T_{q})\|<R(\epsilon).$$
On the other hand, by~\eqref{Feta1}, 
$$\|\eta_{i_0}^{-1}(w^0)\|^2=w^0_2+R(\epsilon)^2\ \text{ and }\ \|\eta_{i_{q-1}}^{-1}(\widetilde  w^q)\|^2=\widetilde  w^q_2+R(\epsilon)^2.$$
Hence 
\begin{equation}\label{r1}w^0_2<0 \ \text{ and }\ \widetilde  w^q_2<0.\end{equation}
In light of~\eqref{Feta1},~(b) and the fact $\widetilde T_k=T_k\in [a,b]=I(\gamma)$ for $k=1,\dots,q-1,$ we have 
$$z^k_2+R(\epsilon)^2=\|\eta_{i_k}^{-1}(z^k)\|^2=\|\gamma(\widetilde T_k)\|^2\geqslant R(\epsilon)^2.$$
So $z^k_2\geqslant 0.$
By construction, it follows that 
\begin{equation}\label{r2}w_2^k<0\ \text{ for }\ k=1,\dots,q-1.\end{equation}
Hence 
$$\widetilde w_2^k+R(\epsilon)^2=\|\eta_{i_{k-1}}^{-1}(\widetilde w^k)\|^2=\|\eta_{i_{k}}^{-1}(w^k)\|^2=w_2^k+R(\epsilon)^2<R(\epsilon)^2,$$
where the first and third equalities follows from~\eqref{Feta1}. This yields
\begin{equation}\label{r3}\widetilde w_2^k<0\ \text{ for }\ k=1,\dots,q-1.\end{equation}
Combining~\eqref{r1},~\eqref{r2} and~\eqref{r3}, we get
\begin{equation}\label{ww}w_{k-1}, \widetilde w^{k}\in B_{x_{i_{k-1}}}\cap\{u\in\mathbb R^n:\ u_2<0\} \ \text{ for }\ k=1,\dots,q.\end{equation}

As $\gamma\in\mathcal A(R(\epsilon)+\epsilon',\epsilon),$ we get
$$f(\gamma(\widetilde T_{0}))< c+\epsilon\ \text{ and }\ f(\gamma(\widetilde T_{q}))< c+\epsilon.$$ 
On the other hand, by~\eqref{Feta1} and the facts $\gamma(\widetilde T_{0})\in W_{i_0}$ and $\gamma(\widetilde T_{q})\in W_{i_{q-1}}$ (by~\eqref{0q}), we get
$$f(\gamma(\widetilde T_{0}))=w^0_1+f(x^{i_0})\ \text{ and }\ f(\gamma(\widetilde T_{q}))=\widetilde w^q_1+f(x^{i_{q-1}}).$$ 
Thus 
\begin{equation}\label{f1}w^0_1<c+\epsilon-f(x^{i_0}) \ \text{ and }\ \widetilde w^q_1<c+\epsilon-f(x^{i_{q-1}}).\end{equation}
In view of~\eqref{Feta1}, for $k=1,\dots,q-1,$ we have 
$$f(\eta_{i_{k}}^{-1}(w^k))=w_1^k+f(x^{i_k})=z_1^k+f(x^{i_{k}})=f(\eta_{i_k}^{-1}(z^k))=f(\gamma(\widetilde T_k))<c+\epsilon.$$ 
So
\begin{equation}\label{f2}w_1^k<c+\epsilon-f(x^{i_{k}}) \ \text{ for }\ k=1,\dots,q-1.\end{equation}
In addition,
$$\widetilde w_1^k+f(x^{i_{k-1}})=f(\eta_{i_{k-1}}^{-1}(\widetilde w^k))=f(\eta_{i_{k}}^{-1}(w^k))<c+\epsilon.$$
This implies
\begin{equation}\label{f3}\widetilde w_1^k<c+\epsilon-f(x^{i_{k-1}}) \ \text{ for }\ k=1,\dots,q-1.\end{equation}
From~\eqref{f1},~\eqref{f2} and~\eqref{f3}, we get
\begin{equation}\label{ff}w_{k-1},\widetilde w^{k}\in B_{x_{i_{k-1}}}\cap\{u\in\mathbb R^n:\ u_1<c+\epsilon-f(x^{i_{k-1}})\} \ \text{ for }\ k=1,\dots,q.\end{equation}
Now for $k=1,\dots,q$, by~\eqref{ww} and~\eqref{ff}, we get
$$w^{k-1},\widetilde w^{k}\in B_{x_{i_{k-1}}}\cap \{u\in\mathbb R^n:\ u_1<c+\epsilon-f(x^{i_{k-1}}),\ u_2<0\}.$$
By convexity, $B_{x_{i_{k-1}}}\cap \{u\in\mathbb R^n:\ u_1<c+\epsilon-f(x^{i_{k-1}}),\ u_2<0\}$ also contains the segment $[w^{k-1},\widetilde w^{k}]$ joining $w^{k-1}$ and $\widetilde w^{k}.$ 
Clearly for any $w=(w_1,\dots,w_n)\in[w^{k-1},\widetilde w^{k}],$ we have 
$$w_1<c+\epsilon-f(x^{i_{k-1}})\ \text{ and }\ w_2<0.$$ 
This and~\eqref{Feta1} yield
$$f(\eta_{i_{k-1}}^{-1}(w))=w_1+f(x^{i_{k-1}})<c+\epsilon\ \text{ and }\ \|\eta_{i_{k-1}}^{-1}(w)\|=\sqrt{w_2+R(\epsilon)^2}<R(\epsilon).$$
Consequently $\|\widetilde\gamma(t)\|<R(\epsilon)$ for any $t\in(\widetilde T_0,\widetilde T_q)$ and so this holds for all $t\in[0,1]$. 
This contradiction ends the proof of the lemma.
\end{proof}
Now we are in position to finish the proof of Proposition~\ref{Prop2}. 
\begin{proof}[Proof of Proposition~\ref{Prop2}]
For each integer $k>0$ large enough such that $\epsilon=\frac{1}{k}$ satisfies the assumptions of Lemma~\ref{1}, in view of Lemma~\ref{key}, there are $x^k\in X_k:=f^{-1}[c,c+\frac{1}{k}]\cap\mathbb S^{n-1}_{R\left(\frac{1}{k}\right)}$ and $v^k\in\partial f(x^k)$ such that $x^k$ and $v^k$ are linear dependent.
Since $R\left(\frac{1}{k}\right)\to+\infty$ as $k\to+\infty$, we have $x^k\to \infty$. 
Furthermore, it is clear that $f(x^k)\to c.$
Consequently $c\in T_\infty(f).$
\end{proof}


\begin{thebibliography}{10}

\bibitem{Ambrosetti2007}
A.~Ambrosetti and A.~Malchiodi.
\newblock {\em Nonlinear Analysis and Semilinear Elliptic Problems}.
\newblock Cambridge University Press, Cambridge, 2007.

\bibitem{Ambrosetti1973}
A.~Ambrosetti and P.~H. Rabinowitz.
\newblock Dual variational methods in critical point theory and applications.
\newblock {\em J. Functional Analysis}, 14:349--381, 1973.

\bibitem{Broughton1988}
Broughton.
\newblock Milnor numbers and the topology of polynomial hypersurfaces.
\newblock {\em Invent. Math.}, 92:217--241, 1988.

\bibitem{Butler1988}
G.~J. Butler, J.~G. Timourian, and C.~Viger.
\newblock The rank theorem for locally {{L}}ipschitz continuous functions.
\newblock {\em Canad. Math. Bull.}, 31(2):217--226, 1988.

\bibitem{Chang1981}
K.~C. Chang.
\newblock Variational methods for nondifferentiable functionals and their
  applications to partial differential equations.
\newblock {\em J. Math. Anal. Appl.}, 80(1):102--129, 1981.

\bibitem{Clarke1976}
F.~H. Clarke.
\newblock A new approach to {{L}}agrange multipliers.
\newblock {\em Math. Oper. Res.}, 1(2):165--174, 1976.

\bibitem{Clarke1981}
F.~H. Clarke.
\newblock Generalized gradients of {{L}}ipschitz functionals.
\newblock {\em Adv. in Math.}, 40(1):52--67, 1981.

\bibitem{Clarke1990}
F.~H. Clarke.
\newblock {\em Optimization and nonsmooth analysis}.
\newblock Classics in Applied Mathematics, SIAM, Philadelphia, PA, 1990.

\bibitem{Coleman2012}
R.~Coleman.
\newblock {\em Calculus on normed vector spaces}.
\newblock Springer, New York, 2012.

\bibitem{Acunto2000}
D.~D'Acunto.
\newblock Valeurs critiques asymptotiques d'une fonction d{{\'e}}finissable
  dans une structure o-minimale.
\newblock {\em Ann. Polon. Math.}, 75(1):35--45, 2000.

\bibitem{HaHV1984}
H.~V. H{{\`a}} and D.~T. L{{\^e}}.
\newblock Sur la topologie des polyn{{\^o}}mes complexes.
\newblock {\em Acta Math. Vietnam.}, 9(1):21--32, 1984.

\bibitem{HaHV2008-2}
H.~V. H\`a and T.~S. Ph\d{a}m.
\newblock Global optimization of polynomials using the truncated tangency
  variety and sums of squares.
\newblock {\em SIAM J. Optim.}, 19(2), 2008.

\bibitem{HaHV2008-1}
H.~V. H\`a and T.~S. Ph\d{a}m.
\newblock On the {{\L}}ojasiewicz exponent at infinity of real polynomials.
\newblock {\em Ann. Polon. Math.}, 94(3), 2008.

\bibitem{HaHV2009-1}
H.~V. H\`a and T.~S. Ph\d{a}m.
\newblock Solving polynomial optimization problems via the truncated tangency
  variety and sums of squares.
\newblock {\em J. Pure Appl. Algebra}, 213:2167--2176, 2009.

\bibitem{HaHV2017}
H.~V. H\`a and T.~S. Ph\d{a}m.
\newblock {\em Genericity in polynomial optimization}, volume~3 of {\em Series
  on Optimization and Its Applications}.
\newblock World Scientific, Singapore, 2017.

\bibitem{Jabri2003}
Y.~Jabri.
\newblock {\em The Mountain Pass Theorem}.
\newblock Cambridge Univ. Press, Cambridge, 2003.

\bibitem{Kim2019}
D.~S. Kim, T.~S. Ph\d{a}m, and N.~V. Tuyen.
\newblock On the existence of {{P}}areto solutions for polynomial vector
  optimization problems.
\newblock {\em Math. Program. Ser. A}, 177(1--2):321--341, 2019.

\bibitem{Nirenberg1981}
l.~Nirenberg.
\newblock Variational and topological methods in nonlinear problems.
\newblock {\em Bull. Amer. Math. Soc. (N.S.)}, 4(3):267--302, 1981.

\bibitem{TaLL1998}
T.~L. Loi and A.~Zaharia.
\newblock Bifurcation sets of functions definable in o-minimal structures.
\newblock {\em Illinois J. Math.}, 43(3):449--457, 1998.

\bibitem{Parusinski1995}
A.~Parusi{{\'n}}ski.
\newblock On the bifurcation set of complex polynomial with isolated
  singularities at infinity.
\newblock {\em Compositio Math.}, 97(3):369--384, 1995.

\bibitem{Paunescu1997}
L.~P{\u a}unescu and A.~Zaharia.
\newblock On the {{\L}}ojasiewicz exponent at infinity for polynomial
  functions.
\newblock {\em Kodai Math. J.}, 20:269--274, 1997.

\bibitem{Paunescu2000}
L.~P{\u a}unescu and A.~Zaharia.
\newblock Remarks on the {{M}}ilnor fibration at infinity.
\newblock {\em manuscripta math.}, 103:351--361, 2000.

\bibitem{Petersen}
P.~Petersen.
\newblock {\em Manifold Theory}.
\newblock https://www.math.ucla.edu/$\sim$petersen/manifolds.pdf.

\bibitem{PHAMTS2020}
T.~S. Ph\d{a}m.
\newblock Local minimizers of semi-algebraic functions from the viewpoint of
  tangencies.
\newblock {\em SIAM J. Optim.}, 30(3):1777--1794, 2020.

\bibitem{Rabier1997}
P.~Rabier.
\newblock Ehresmann fibrations and {{P}}alais-{{S}}male conditions for
  morphisms of {{F}}insler manifolds.
\newblock {\em Ann. of Math.}, 146(3):647--691, 1997.

\bibitem{Rabinowitz1986}
P.~Rabinowitz.
\newblock {\em Minimax Methods in Critical Point Theory with Applications to
  Differential Equations}.
\newblock CBMS Reg. Conf. Ser. in Math. No. 65, AMS, Providence, 1986.

\bibitem{Siersma1995}
D.~Siersma and M.~Tib{{\u{a}}}r.
\newblock Singularities at infinity and their vanishing cycles.
\newblock {\em Duke Math. J.}, 80(3):771--783, 1995.

\bibitem{Struwe1996}
M.~Struwe.
\newblock {\em Variational Methods}.
\newblock Springer, Berlin, 1996.

\bibitem{Thom1969}
R.~Thom.
\newblock Ensembles et morphismes stratifi{{\'e}}s.
\newblock {\em Bull. Amer. Math. Soc.}, 75:249--312, 1969.

\bibitem{Dries1998}
L.~van~den Dries.
\newblock {\em Tame topology and o-minimal structures}, volume 248 of {\em
  London Mathematical Society Lecture Note Series}.
\newblock Cambridge University Press, Cambridge, 1998.

\bibitem{Willem1996}
M.~Willem.
\newblock {\em Minimax Theorems}.
\newblock Birkh{\"a}user, Basel, 1996.

\end{thebibliography}

\end{document}